\documentclass[12pt,onecolumn,draftcls]{IEEEtran}
\usepackage{subcaption}
\usepackage{amsmath}
\usepackage{amsbsy}
\usepackage{txfonts}
\usepackage{stmaryrd}
\usepackage{enumerate}
\usepackage{graphicx}
\usepackage{color}
\newtheorem{theorem}{Theorem}[section]

\newtheorem{Lemma}[theorem]{Lemma}
\newtheorem{rmk}{Remark}
\newtheorem{proof}[theorem]{Proof}

\newenvironment{Proof}[1][Proof]{\begin{trivlist}
\item[\hskip \labelsep {\bfseries #1}]}{\end{trivlist}}
\newtheorem{proposition}[theorem]{Proposition}

\title{Control of Homodirectional and General Heterodirectional Linear Coupled Hyperbolic PDEs}

\author{Long Hu\thanks{L. Hu is with Sorbonne Universit\'{e}s, UPMC Univ Paris 06, UMR 7598, Laboratoire Jacques-Louis Lions, F-75005, Paris, France. School of Mathematical
Sciences, Fudan University, Shanghai 200433, China. E-mail: \texttt{hu@ann.jussieu.fr},\ \ \texttt{hul10@fudan.edu.cn}. This author
was supported by the China Scholarship Council for Ph.D. study at UPMC (No. 201306100081) and was partially supported by ERC advanced grant 266907
(CPDENL) of the 7th Research Framework
Programme (FP7). }, 
	Florent Di Meglio\thanks{F. Di Meglio is with MINES ParisTech, PSL Research University, CAS - Centre automatique et syst\`emes, 60 bd St Michel, 75006 Paris, France.
        {\texttt{florent.di\_meglio{\@}mines-paristech.fr}}}, 
        Rafael Vazquez\thanks{Rafael Vazquez is with the Department of Aerospace Engineering, Universidad de Sevilla, Camino de los Descubrimiento s.n., 41092 Sevilla, Spain.
        {\texttt{rvazquez1@us.es}}}, 
        Miroslav Krstic\thanks{M. Krstic is with the Department of Mechanical and Aerospace Engineering, University of California San Diego, La Jolla, CA 92093-0411, USA.
        {\texttt{krstic@ucsd.edu}}}}
\begin{document}
\maketitle
%
\begin{abstract}
Research on stabilization of coupled hyperbolic PDEs has been dominated by the focus on pairs of counter-convecting (``heterodirectional'') transport PDEs with distributed local coupling and with controls at one or both boundaries. A recent extension allows stabilization using only one control for a system containing an arbitrary number of coupled transport PDEs that convect at different speeds against the direction of the PDE whose boundary is actuated. In this paper we present a solution to the fully general case, in which the number of PDEs in either direction is arbitrary, and where actuation is applied on only one boundary (to all the PDEs that convect downstream from that boundary). To solve this general problem, we solve, as a special case, the problem of control of coupled ``homodirectional'' hyperbolic linear PDEs, where multiple transport PDEs convect in the same direction with arbitrary local coupling. Our approach is  based on PDE backstepping  and yields solutions to stabilization, by both full-state and observer-based output feedback, trajectory planning, and trajectory tracking problems. 
\end{abstract}
\section{Introduction}

\paragraph{Background} Coupled first-order linear hyperbolic systems, typically formulated on a  1-D spatial domain normalized to the interval $(0,1)$, are common in modeling of traffic flow~\cite{Amin2008}, heat exchangers~\cite{Xu2002}, open channel flow~\cite{Coron1999,Halleux2003} or multiphase flow~\cite{DiMeglio2011,Djordjevic2010,Dudret2012}. 

Research on stabilization of such PDEs has been dominated by the focus on pairs of counter-convecting transport PDEs with distributed local coupling. In~\cite{Coron2013} a first solution allowing actuation on only one boundary and permitting coupling coefficients of arbitrary size was presented. A recent extension~\cite{DiMeglio2013} by three of the authors of the present paper allows stabilization using only one control for a system containing an arbitrary number of coupled transport PDEs that convect at different speeds against the direction of the PDE whose boundary is actuated. 

In this paper we present a solution to the fully general case of coupled hyperbolic PDEs. We divide such PDE systems into two categories:
\begin{itemize}
	\item \emph{homodirectional} systems of $m$ transport PDEs, for which all the $m$ transport velocities have the same signs, i.e., all of the PDEs convect in the same direction. Because of the finite length of the spatial domain, these are inherently stable but the coupling between states can cause undesirable transient behaviors and the trajectory planning problem is non-trivial. 
	\item \emph{heterodirectional} systems of $n+m$ transport PDEs, for which there exist at least two transport velocities with opposite signs, i.e., where $m$ PDEs convect in one direction and $n$ PDEs convect in the opposite direction. The coupling between states traveling in opposite directions may cause instability. 
\end{itemize}

In this paper we present control designs for the fully general case of coupled heterodirectional hyperbolic PDEs, allowing the numbers $m$ and $n$ of PDEs in either direction to be arbitrary, and with actuation applied on only one boundary (to all the $m$ PDEs that convect downstream from that boundary). To solve this general problem, we solve, as a special case, the heretofore unsolved problem of control of coupled homodirectional hyperbolic linear PDEs, where multiple transport PDEs convect in the same direction, have possibly distinct speeds, and arbitrary local coupling. 

Our approach is  based on PDE backstepping and yields solutions to stabilization, by both full-state and observer-based output feedback, trajectory planning, and trajectory tracking problems.

\paragraph{Literature} 
Controllability of hyperbolic systems has first been investigated using explicit computation of the solution along the characteristic curves in the framework of $C^1$ norm~\cite{Greenberg1984,Li1994,Qin1985}. Later, the so-called Control Lyapunov Functions methods emerged, enabling the design of dissipative boundary conditions for nonlinear hyperbolic systems in the context of both $C^1$ norm and $H^2$ norm~\cite{Coron2007,Coron2008,Coron2014}. 
Further, using Lyapunov functions method, sufficient boundary conditions for the exponential stability of linear~\cite{Diagne2012} or nonlinear~\cite{Gugat2011,Gugat2011a} hyperbolic systems of balance laws have been derived.  All of these results impose restrictions on the magnitude of the coupling coefficients, which are responsible for potential instabilities. 

In~\cite{Coron2013}, a full-state feedback control law, with actuation only on one end of the domain, which achieves $H^2$ exponential stability of closed-loop 2--state heterodirectional linear and quasilinear hyperbolic systems is derived using a backstepping method. With a similar backstepping transformation, an output-feedback controller is designed in~\cite{DiMeglio2013} for heterodirectional systems with $m=1$ (controlled) negative velocity and $n$ (arbitrary) positive ones. These results hold regardless of the (bounded) magnitude of the coupling coefficients. Unfortunately, the method presented  in \cite{Coron2013,DiMeglio2013} can not be extended to the case $m>1$.

\paragraph{Contribution}
The first step towards this paper's general solution for $m>1$ was presented (but not published as a paper) in~\cite{VazKrs2010} for $m=2$ and $n=0$. In conference paper~\cite{Hu2015}, an extension to $m=2$ and $n=1$ is achieved. 

The contribution of this article is two-fold. For $(n+m)$--state heterodirectional systems, we derive a stabilizing boundary feedback law that ensures finite-time convergence of all the states to zero. For homodirectional systems (for which stability is not an issue), we design a boundary control law ensuring tracking of a given reference trajectory at the uncontrolled boundary. 

Both designs rely on the backstepping approach. A particular choice of the target system, featuring a cascade structure similar to~\cite[Section~3.5]{Coron2013}, enables the use of a classical Volterra integral transformation. Well-posedness of the system of kernel equations, which is the main technical challenge of this paper, is proved by a method of successive approximations using a novel recursive bound. 

In the case of heterodirectional systems, the approach yields a full-state feedback law that would necessitate full distributed measurements to be implemented, which is not realistic in practice. For this reason, we derive an observer relying on measurements of the states at a single boundary (the anti-controlled one). Along with the full-state feedback law, this yields an output feedback controller amenable to implementation. 

\paragraph{Organization} In Section~\ref{sec:system} we introduce the model equations. In Section~\ref{sec:stabilization} we present the stabilization result for heterodirectional systems: the target system is presented in~Section~\ref{sec:target} while the backstepping transformation is derived in Section~\ref{sec:transformation}. The design is summarized in Section~\ref{sec:main}. In Section~\ref{sec:observer} we present the boundary observer design. In Section~\ref{sec:motion_planning} we present the motion planning result for homodirectional systems. Section~\ref{sec:wellposedness} contains the main technical difficulty of the paper, i.e. the proof of well-posedness of the backstepping transformation. We conclude in Section~\ref{sec:outlook} by discussing open problems. 

\section{System description}\label{sec:system}
We consider the following general linear hyperbolic system
\begin{align}
	u_t (t,x) + \Lambda^+ u_x (t,x) &= \Sigma^{++} u(t,x) + \Sigma^{+-} v(t,x)\label{eq:system1}\\
	v_t (t,x) - \Lambda^- v_x (t,x) &= \Sigma^{-+} u(t,x) + \Sigma^{--} v(t,x)\label{eq:system2}
\end{align}
with the following boundary conditions
\begin{align}
	u(t,0)&= Q_0 v(t,0),&v(t,1)&= R_1 u(t,1) + U(t) \label{eq:systemBC}
\end{align}
where
\begin{align}
	u&=\begin{pmatrix}
		u_1&\cdots&u_n
	\end{pmatrix}^T,&v&=\begin{pmatrix}
		v_1&\cdots&v_m
	\end{pmatrix}^T
\end{align}
\begin{align}
	\Lambda^{+}&=\begin{pmatrix}
		\lambda_1&&0\\
		&\ddots&\\
		0&&\lambda_n
	\end{pmatrix},&\Lambda^{-}&=\begin{pmatrix}
		\mu_1&&0\\
		&\ddots&\\
		0&&\mu_m
	\end{pmatrix}
	\label{eq:coefLambda}
\end{align}
with
\begin{align}
	-\mu_1<\cdots<-\mu_m<0<\lambda_1&\leq\cdots\leq\lambda_n \label{eq:speeds}
\end{align}
and 
\begin{align}
	\Sigma^{++}&=\left\{\sigma^{++}_{ij}\right\}_{1\leq i \leq n, 1\leq j \leq n},&\Sigma^{+-}&=\left\{\sigma^{+-}_{ij}\right\}_{1\leq i \leq n, 1\leq j \leq m},\\
	\Sigma^{-+}&=\left\{\sigma^{-+}_{ij}\right\}_{1\leq i \leq m, 1\leq j \leq n},&\Sigma^{--}&=\left\{\sigma^{--}_{ij}\right\}_{1\leq i \leq m, 1\leq j \leq m} \label{eq:coefSigma}
\end{align}
\begin{align}
	Q_0&=\left\{q_{ij}\right\}_{1\leq i \leq n, 1\leq j \leq m},&R_1&=\left\{\rho_{ij}\right\}_{1\leq i \leq m, 1\leq j \leq n},\\U(t)& = \begin{pmatrix}
		U_1(t)&\cdots&U_m(t)
	\end{pmatrix}^T
\end{align}
\begin{rmk}\label{rmk:space}
	We consider here constant coupling coefficients and transport velocities for the sake of readability. The method straightforwardly extends to spatially varying coefficients, with more involved technical developments.   
\end{rmk}
Besides, we also make the following assumption without loss of generality
\begin{align}
	\forall j&=1,...,m&\sigma_{jj}^{--}&=0,
\end{align}
i.e. there are no (internal) diagonal coupling terms for $v$-system. Such coupling terms can be removed using a change of coordinates as presented in, e.g.,~\cite{Coron2013} and~\cite{Hu2015}. This yields spatially-varying coupling terms, which is not an issue in the light of Remark~\ref{rmk:space}.   

\begin{rmk}\label{rmk:isotachic}
If two or more states have the same transport speeds (i.e. $\mu_i=\mu_j$ for some $i\neq j$) we refer to those states as \emph{isotachic}. This case was intentionally avoided in (\ref{eq:speeds}). To deal with isotachic states, we consider the change of coordinates
$\bar v(t,x)=A(x)v(t,x)$.
The matrix $A(x)$ is a block-diagonal matrix, with $A_{ii}=1$ if $\mu_i\neq\mu_j$ for $j\neq i$. If there is a set of $n_i$ isotachic states (i.e.  there is $i$ such that $\mu_j=\mu_i$ for $j=i+1,\hdots,i+n_{i}-1$, then there is in $A(x)$ a corresponding block $B(x)$ of dimension $n_i\times n_i$ in $A(x)$. Each of these $B(x)$ is computed independently for each isotachic set of states. If we call $\Sigma_{iso}$ the matrix of coupling coefficients among these isotachic states (i.e. with coefficients $\sigma_{jk}^{--}$ for $j,k=i,i+1,\hdots,i+n_{i}-1$), then $B(x)$ is computed from the initial value problem $BÕ(x)=1/\mu_i B(x) \Sigma_{iso}$, $B(0)=I_{n_i\times n_i}$. It is easy to see that this transformation is invertible, since one can define a matrix $C(x)$  from $CÕ(x)=1/\mu_i \Sigma_{iso} B(x) $, $C(0)=I_{n_i\times n_i}$. One has that $C(x)$ is the inverse of $B(x)$ as $B(0)C(0)=I_{n_i\times n_i}$ and $\frac{d}{dx} B(x)C(x)=0$. Applying this invertible transformation eliminates the coupling coefficients between isotachic states, but  results in some spatially-varying coupling terms, which is not an issue as explained in Remark~\ref{rmk:space}.\end{rmk}

\section{Stabilization of heterodirectional systems}\label{sec:stabilization}
In this section, we derive a stabilizing feedback law for the general$~(n+m)$--state system. Notice that this is interesting only in the case~$n \neq 0$, since instability arises from coupling between states traveling in opposite directions. Following the backstepping approach, we seek to map system~\eqref{eq:system1}--\eqref{eq:systemBC} to a target system with desirable stability properties using an invertible Volterra transformation. 
\subsection{Target system}\label{sec:target}
\subsubsection{Target system design}
We map system~\eqref{eq:system1}--\eqref{eq:systemBC} to the following target system
\begin{align}
	\notag&\alpha_t (t,x) + \Lambda^+ \alpha_x (t,x) = \Sigma^{++} \alpha(t,x) + \Sigma^{+-} \beta(t,x)\\
	&\ \ +\int_0^x C^+(x,\xi)\alpha(\xi)d\xi+\int_0^x C^-(x,\xi)\beta(\xi)d\xi \label{eq:target1}\\
	&\beta_t (t,x) - \Lambda^- \beta_x (t,x) =   G(x) \beta(0) \label{eq:target2}
\end{align}
with the following boundary conditions
\begin{align}
	\alpha(t,0)&= Q_0 \beta(t,0),&\beta(t,1)&= 0\label{eq:targetBC}
\end{align}
where $C^+$ and $C^-$ are $L^{\infty}$ matrix functions on the domain
 \begin{align}
	\mathcal{T}&=\left\{0\leq \xi\leq x\leq 1\right\},
\end{align}
while~$G\in L^{\infty}(0,1)$ is a lower triangular matrix with the following structures 
\begin{align}\label{matrixG}
	G(x)&=\begin{pmatrix}
		0&\cdots&\cdots&0\\
		g_{2,1}(x)&\ddots&\ddots&\vdots\\
		\vdots&\ddots&\ddots&\vdots\\
		g_{m,1}(x)&\cdots&g_{m,m-1}(x)&0
	\end{pmatrix}.
\end{align}
The coefficients of~$C^+$, $C^-$ and~$G$ will be determined in section \ref{sec:transformation}.
\subsubsection{Stability of the target system}
The following lemma asseses the finite-time stability of the target system. 
\begin{Lemma}\label{lem:stabTarget}
	Consider system~\eqref{eq:target1},\eqref{eq:target2} with boundary conditions~\eqref{eq:targetBC}. Its zero equilibrium is reached in finite time~$t=t_F$, where
	\begin{align}\label{finitetime}
	t_F:=\frac{1}{\lambda_1}+\sum_{j=1}^m\frac{1}{\mu_j}.
	\end{align}
\end{Lemma}
\begin{proof}
Noting \eqref{eq:target2}-\eqref{eq:targetBC} and \eqref{matrixG}, we find that the~$\beta$--system is in fact a cascade system, which allows us to explicitely solve it by recursion as follows. The explicit solution of $\beta_1$ is given by
\begin{align}\label{2.8z}
\begin{aligned}
\beta_1(t,x)=\begin{cases}
\beta_{1}(0,x+\mu_1 t) &\text{if } t< \frac{1-x}{\mu_1},\\
0&\text{if } t\geq  \frac{1-x}{\mu_1}.
\end{cases}
\end{aligned}
\end{align}
Notice in particular that~$\beta_1$ is identically zero for~$t\geq \mu_1^{-1}$. From the time~$t\geq \mu_1^{-1}$ on, we have that $\beta_2(t,x)$ satisfies the following equation
\begin{align}\label{2.9z}
\beta_{2t}(t,x)-\mu_2\beta_{2x}(t,x)=0.
\end{align}
Similarly, by expressing the solution along the characteristic lines, one obtains that
\begin{align}
 \beta_2(t,x)&\equiv 0\label{eq:gamma_20} \quad \forall  t \geq \mu_1^{-1}+\mu_2^{-1}.
\end{align}
Thus, by mathematical induction, one can easily get that $\beta_j(j=1,\cdots,m)$  vanishes after
\begin{align}
t=\sum_{k=1}^j\frac{1}{\mu_k}.
\end{align}
This yields that 
\begin{align}
\beta(t,x)\equiv 0,\ \ t>\sum\limits_{j=1}^m\frac{1}{\mu_j}.
\end{align}
When $t>\sum\limits_{j=1}^m\frac{1}{\mu_j}$, the $\alpha$--system becomes
\begin{align}
	\alpha_t (t,x) + \Lambda^+ \alpha_x (t,x) = \Sigma^{++} \alpha(t,x) +\int_0^x C^+(x,\xi)\alpha(\xi)d\xi\label{neweq:target1}
\end{align}
with the boundary conditions
\begin{align}
	\alpha(t,0)&= 0.\label{neweq:targetBC}
\end{align}
Since there are no zero transport velocities for the~$\alpha$--system (see \eqref{eq:speeds}), we may change the status of~$t$ and~$x$, and Equations~\eqref{neweq:target1} can be rewritten as
\begin{align}
	\alpha_x (t,x)+(\Lambda^+)^{-1}\alpha_t (t,x)  = (\Lambda^+)^{-1}\Sigma^{++} \alpha(t,x) +\int_0^x (\Lambda^+)^{-1} C^+(x,\xi)\alpha(\xi)d\xi\label{neweq:target2}
\end{align}
with the initial condition \eqref{neweq:targetBC}. Then by the uniqueness of the system \eqref{neweq:targetBC},\eqref{neweq:target2}, and noting the order of the transport speeds of the~$\alpha$--system (see \eqref{eq:speeds}),  this yields that~$\alpha$ identically vanishes for
\begin{align}
 t\geq \frac{1}{\lambda_1}+\sum_{j=1}^m\frac{1}{\mu_j} 
\end{align}
This concludes the proof.
\end{proof}
\subsection{Backstepping transformation}\label{sec:transformation}
To map system~\eqref{eq:system1}--\eqref{eq:systemBC} to the target system~\eqref{eq:target1}--\eqref{eq:targetBC}, we consider the following backstepping (Volterra) transformation
\begin{align}
\alpha(t,x)=& u(t,x)\\
\beta(t,x) =& v(t,x) - \int_0^x \left[K(x,\xi)u(\xi) + L(x,\xi)v(\xi)\right]d\xi\label{eq:BS}	
\end{align}
where the kernels to be determined~$K$ and~$L$ are defined on the triangular domain~$\mathcal{T}$.
Deriving~\eqref{eq:BS} with respect to space and time,
plugging into the target system equations and  noticing that~$\beta(t,0)\equiv v(t,0)$ yields the following system of  kernel equations
\begin{align}
	0 =& K(x,x)\Lambda^+ + \Lambda^- K(x,x) + \Sigma^{-+}\label{eq:hypotenuseplus}\\
	0 =& \Lambda^- L(x,x)-L(x,x)\Lambda^- + \Sigma^{--} \label{eq:hypotenuseminus}\\
	0=&K(x,0)\Lambda^+ Q_0+G(x)-L(x,0)\Lambda^- \label{eq:x0boundary}\\
	\notag0=&\Lambda^-K_x(x,\xi) - K_\xi(x,\xi)\Lambda^+ \\
	&-K(x,\xi)\Sigma^{++}-L(x,\xi)\Sigma^{-+}\\
	\notag0=&\Lambda^-L_x(x,\xi) + L_\xi(x,\xi)\Lambda^-\\
	&-L(x,\xi)\Sigma^{--}-K(x,\xi)\Sigma^{+-}\label{eq:kernelEDPmin}
\end{align}
and yields the following equations for~$C^-(x,\xi)$ and~$C^+(x,\xi)$
\begin{align}
	C^-(x,\xi)&=L(x,\xi)+\int_\xi^x C^-(x,s)L(s,\xi)d\xi\label{eq:Cminus}\\
	C^+(x,\xi)&=K(x,\xi)+\int_\xi^x C^-(x,s)K(s,\xi)d\xi\label{eq:Cplus}
\end{align}
\begin{rmk}
	For each~$x\in[0,1]$, Equation~\eqref{eq:Cminus} is a Volterra equation of the second kind on~$[0,x]$ with~$C^-(x,\cdot)$ as the unknown. Besides, Equation~\eqref{eq:Cplus} explicitly gives~$C^+(x,\xi)$ as a function of~$C^-(x,\xi)$ and~$K(x,\xi)$. Therefore, provided the kernels~$K$ and~$L$ are well-defined and bounded, so are~$C^+$ and~$C^-$.
\end{rmk}
Developing equations~\eqref{eq:hypotenuseplus}--\eqref{eq:kernelEDPmin} leads to the following set of kernel PDEs

{\underline{for $1\leq i \leq m$, $1\leq j \leq n$}}
\begin{align}
	\mu_i \partial_x K_{ij}(x,\xi) - \lambda_j \partial_\xi K_{ij}(x,\xi) =
	\sum\limits_{k=1}^n\sigma^{++}_{kj}K_{ik}(x,\xi)+\sum\limits_{p=1}^m \sigma^{-+}_{pj}L_{ip}(x,\xi)\label{eq:developedKernelK}
	\end{align}
	{\underline{for $1\leq i \leq m$, $1\leq j \leq m$}}
	\begin{align}
	\mu_i \partial_x L_{ij}(x,\xi) + \mu_j \partial_\xi L_{ij}(x,\xi) =	\sum\limits_{p=1}^m\sigma^{--}_{pj} L_{ip}(x,\xi)+\sum\limits_{k=1}^n \sigma^{+-}_{kj}K_{ik}(x,\xi)\label{eq:developedKernelL}
\end{align}
along with the following set of boundary conditions
\begin{align}
	\forall 1&\leq i \leq m, 1\leq j \leq n,&
		K_{ij} (x,x)&= - \cfrac{\sigma^{-+}_{ij}}{\mu_i+\lambda_j} \stackrel{\Delta}{=}k_{ij} \label{eq:hypotenuseK}\\
\forall 1&\leq i,j \leq m,  i \neq j,& 	L_{ij} (x,x)&= - \cfrac{\sigma^{--}_{ij}}{\mu_i-\mu_j}\stackrel{\Delta}{=}l_{ij} \label{eq:hypotenuseL}\\
\forall 1&\leq i \leq j \leq m,	&	\mu_j L_{ij}(x,0)& =\sum\limits_{k=1}^n \lambda_k K_{ik}(x,0)q_{k,j}\label{eq:modifx0boundary}
\end{align}
To ensure well-posedness of the kernel equations, 
we add the following artificial boundary conditions for $L_{ij}(i>j)$
\begin{align}\label{eq:artificialboundary}
L_{ij}(1,\xi)=l_{ij}, \ \text{for} \ \ 1&\leq j<i \leq m
\end{align}
While the~$g_{ij}$, for $1\leq j<i \leq n$, are given by
\begin{align}\label{eq:gijfunctionLij}
	g_{ij}(x)&=\mu_j L_{ij}(x,0) - \sum\limits_{p=1}^n \lambda_p q_{pj}  K_{ip}(x,0)
\end{align} 
provided the~$K$ and~$L$ kernels are properly defined by~\eqref{eq:developedKernelK}--\eqref{eq:artificialboundary}, which we prove in the next section. 
\begin{rmk}
The choice of imposing~\eqref{eq:artificialboundary} as the boundary condition for $L_{ij}(1\leq j<i\leq m)$, on the boundary~$x=1$ is arbitrary and was designed to ensure continuity of some of the kernels. This degree of freedom in the control design had never appeared in previous backstepping designs for hyperbolic system~\cite{Coron2013,DiMeglio2013}. The impact of the boundary values of~$L_{ij}$, $1\leq j<i\leq m$ on the transient behavior of the closed-loop system remains an open question, out of the scope of this article. 
\end{rmk}
\begin{rmk}
If there are isotachic states, and the transformation explained in Remark~\ref{rmk:isotachic} is applied, then the $L_{ij}$ kernels for $i,j$ corresponding to isotachic states ($\mu_i=\mu_j$) have all boundary conditions of the type (\ref{eq:modifx0boundary}) instead of (\ref{eq:hypotenuseL})---which would become singular---or (\ref{eq:artificialboundary}). The results that follow do not change, but we have omitted the case for the sake of brevity.
\end{rmk}

The well-posedness of the target system equations is assessed in the following Theorem.
\begin{theorem}\label{the:wellposednessKandL}
Consider system~\eqref{eq:developedKernelK}--\eqref{eq:artificialboundary}. There exists a unique solution $K$ and $L$ in $L^\infty(\mathcal{T})$. Moreover, all the boundary traces for the $K$-kernel and $L$-kernel are functions of $L^{\infty}(0,1)$.
\end{theorem}
The proof of this Theorem is the main technical difficulty of the paper and is presented in Section~\ref{sec:wellposedness}.
\subsection{Control law and main stabilization result}\label{sec:main}
We are now ready to state the main stabilization result as follows. 
\begin{theorem}\label{the:main}
Consider system~\eqref{eq:system1}-\eqref{eq:system2} with boundary conditions~\eqref{eq:systemBC} and the following feedback control law
\begin{align}
 U(t) = -R_1u(t,1) +\int_0^1 \left[K(1,\xi)u(\xi) + L(1,\xi)v(\xi)\right]d\xi\label{eq:controlLaw}
\end{align}
For any initial condition $(u_0,v_0)\in (L^\infty(0,1))^{(n+m)\times(n+m)}$, the zero equilibrium is reached in finite time~$t=t_F$, where~$t_F$ is given by \eqref{finitetime}.
\end{theorem}
\begin{proof}
First, notice that evaluating transformation~\eqref{eq:BS} at~$x=1$ yields~\eqref{eq:controlLaw}. Besides, rewriting transformation~\eqref{eq:BS} as follows
\begin{align}\label{final kernel}
\left(\begin{array}{c}\alpha(t,x)\\ \beta(t,x) \end{array}\right)=\left(\begin{array}{c}u(t,x)\\ v(t,x) \end{array}\right)
-\int_0^x\left(\begin{array}{cc}0&0\\ K(x,\xi)& L(x,\xi) \end{array}\right)\left(\begin{array}{c}u(t,\xi)\\ v(t,\xi) \end{array}\right) d\xi.
\end{align}
one notices that it is a classical Volterra equation of the second kind. One can check from, e.g.,~\cite{Hochstadt1973} that there exists a unique matrix function $\mathcal{R}\in (L^{\infty}(\mathcal{T}))^{(n+m)\times (n+m)}$ such that
\begin{align}\label{invertible}
\left(\begin{array}{c}u(t,x)\\ v(t,x) \end{array}\right)=\left(\begin{array}{c}\alpha(t,x)\\ \beta(t,x) \end{array}\right)-\int_0^x\mathcal{R}(x,\xi)\left(\begin{array}{c}\alpha(t,\xi)\\ \beta(t,\xi) \end{array}\right) d\xi.
\end{align}
Applying Lemma~\ref{lem:stabTarget} implies that $(\alpha,\beta)$ go to zero in finite time $t = t_F$ , therefore, by \eqref{invertible}, $(u,v)$ also converge to zero in finite time.
\end{proof}

\section{Uncollocated observer design and output feedback controller}\label{sec:observer}
In this section, we derive an observer that relies on the measurement of the~$v$ states at the left boundary, i.e.
\begin{align}
	y(t)&=v(t,0)
\end{align}
Then, using the estimates from the observer along with the control law~\eqref{eq:controlLaw}, we derive an output feedback controller. 
\subsection{Observer design}
The observer equations read as follows 
\begin{align}
	 \hat{u}_t (t,x) + \Lambda^+ \hat{u}_x (t,x) =& \Sigma^{++} \hat{u}(t,x) + \Sigma^{+-} \hat{v}(t,x)-P^+(x)(\hat{v}(t,0)-v(t,0))\label{eq:obs1}\\
	\hat{v}_t (t,x) - \Lambda^- \hat{v}_x (t,x) =& \Sigma^{-+} \hat{u}(t,x) + \Sigma^{--} \hat{v}(t,x)-P^-(x)(\hat{v}(t,0)-v(t,0))\label{eq:obs2}
\end{align}
with the following boundary conditions
\begin{align}
	\hat{u}(t,0)&= Q_0 v(t,0),&\hat{v}(t,1)&= R_1 \hat{u}(t,1) + \hat{u}(t) \label{eq:obsBC}
\end{align}
where~$P^+(\cdot)$ and~$P^-(\cdot)$ have yet to be designed. This yields the following error system
\begin{align}
	\tilde{u}_t (t,x) + \Lambda^+ \tilde{u}_x (t,x) =& \Sigma^{++} \tilde{u}(t,x) + \Sigma^{+-} \tilde{v}(t,x)-P^+(x)\tilde{v}(t,0)\label{eq:errorobs1}\\
	\tilde{v}_t (t,x) - \Lambda^- \tilde{v}_x (t,x) =& \Sigma^{-+} \tilde{u}(t,x) + \Sigma^{--} \tilde{v}(t,x)-P^-(x)\tilde{v}(t,0)\label{eq:errorobs2}
\end{align}
with the following boundary conditions
\begin{align}
	\tilde{u}(t,0)&=0,&\tilde{v}(t,1)&= R_1 \tilde{u}(t,1) \label{eq:errorobsBC}
\end{align}
\begin{rmk}
	One should notice that the output is directly injected at the left boundary, which means potential sensor noise is only filtered throughout the spatial domain. Combining the approach of~\cite{DiMeglio2013} and the cascade structure of~\eqref{eq:target1}--\eqref{eq:targetBC}, we now derive a target system and backstepping transformation to design observer gains~$P^+(\cdot)$ and~$P^-(\cdot)$ that yield finite-time stability of the error system~\eqref{eq:errorobs1}--\eqref{eq:errorobsBC}.
\end{rmk}
\subsection{Target system and backstepping tranformation}
We map system~\eqref{eq:errorobs1}--\eqref{eq:errorobsBC} to the following target system
\begin{align}
	\tilde{\alpha}_t (t,x) + \Lambda^+ \tilde{\alpha}_x (t,x) =& \Sigma^{++} \tilde{\alpha}(t,x) + \int_0^x D^+(x,\xi)\tilde{\alpha}(\xi)d\xi\label{eq:targetobs1}\\
	\tilde{\beta}_t (t,x) - \Lambda^- \tilde{\beta}_x (t,x) =&  \Sigma^{-+}\tilde{\alpha}(t,x) +\int_0^x D^-(x,\xi)\tilde{\alpha}(\xi)d\xi  \label{eq:targetobs2}
\end{align}
with the following boundary conditions
\begin{align}
	\tilde{\alpha}(t,0)&= 0,&\tilde{\beta}(t,1)&= R_1\tilde{\alpha}(t,1)-\int_0^1 H(\xi)\tilde{\beta}(\xi)d\xi\label{eq:targetobsBC}
\end{align}
where $D^+$ and $D^-$ are $L^{\infty}$ matrix functions on the domain $\mathcal{T}$ and~$H\in L^{\infty}(0,1)$ is an upper triangular matrix with the following structure 
\begin{align}\label{eq:H}
	H(x)&=\begin{pmatrix}
		0&h_{1,2}(x)&\cdots&h_{1,m}(x)\\
		\vdots&\ddots&\ddots&\vdots\\
		\vdots&\ddots&\ddots&h_{m-1,m}(x)\\
		0&\cdots&\cdots&0
	\end{pmatrix}
\end{align}
all of which have yet to be determined.
\begin{proposition}
\label{prop:stabtargetObs}
	The solutions of system~\eqref{eq:targetobs1}--\eqref{eq:targetBC} converge to zero in finite time. More precisely, one has
	\begin{align}
			\forall t\geq t_F,\quad \tilde{\alpha}&\equiv \tilde{\beta} \equiv 0
	\end{align}
	where~$t_F$ is defined by~\eqref{finitetime}.
\end{proposition}
\begin{Proof}
	The system consists in a cascade of the~$\tilde{\alpha}$--system (that has zero input at the left boundary) into the~$\tilde{\beta}$--system. Further, the~$\tilde{\beta}$ is a cascade of its slow states into its fast states. The rigorous proof follows the same steps that the proof of Lemma~\ref{lem:stabTarget} and is therefore omitted here. 
\end{Proof}

To map system~\eqref{eq:errorobs1}--\eqref{eq:errorobsBC} to the target system~\eqref{eq:targetobs1}--\eqref{eq:targetobsBC}, we consider the following backstepping (Volterra) transformation
\begin{align}
\tilde{u}(t,x)&=\tilde{\alpha}(t,x)+\int_0^xM(x,\xi)\tilde{\beta}(\xi)d\xi\label{eq:BS1}\\
\tilde{v}(t,x)&=\tilde{\beta}(t,x)+\int_0^x N(x,\xi)\tilde{\beta}(\xi)d\xi\label{eq:BS2}	
\end{align}
where the kernels to be determined~$M$ and~$N$ are defined on the triangular domain~$\mathcal{T}$.
Deriving~\eqref{eq:BS1},\eqref{eq:BS2} with respect to space and time yields
the following kernel equations

{\underline{for $1\leq i \leq n$, $1\leq j \leq m$}}
\begin{align}
	\lambda_i \partial_x M_{ij}(x,\xi) - \mu_j \partial_\xi M_{ij}(x,\xi) =
	\sum\limits_{k=1}^n\sigma^{++}_{ik}M_{kj}(x,\xi)+\sum\limits_{p=1}^m \sigma^{+-}_{ip}N_{pj}(x,\xi)\label{eq:developedKernelM}
	\end{align}
	{\underline{for $1\leq i \leq m$, $1\leq j \leq m$}}
	\begin{align}
	\mu_i \partial_x N_{ij}(x,\xi) + \mu_j \partial_\xi N_{ij}(x,\xi) = \sum\limits_{k=1}^n\sigma^{-+}_{ik} M_{kj}(x,\xi)+\sum\limits_{p=1}^m \sigma^{--}_{ip}N_{pj}(x,\xi)\label{eq:developedKernelN}
\end{align}
along with the following set of boundary conditions
\begin{align}
\forall 1&\leq i \leq m, \ 1\leq j \leq n&
	M_{ij} (x,x)&= \cfrac{\sigma^{+-}_{ij}}{\mu_i+\lambda_j} \stackrel{\Delta}{=}m_{ij} \label{eq:hypotenuseM}\\
	\forall 1& \leq i,j \leq m, \ i \neq j,& N_{ij} (x,x)&= 0\label{eq:hypotenuseN}\\
	\intertext{besides, evaluating~\eqref{eq:BS1},\eqref{eq:BS2} at~$x=1$ yields}
	 \forall 1&\leq j \leq i \leq m &	N_{ij}(1,x) &=\sum\limits_{k=1}^n \rho_{ik} M_{kj}(1,x)\label{eq:modifx1boundaryobs}
\intertext{
To ensure well-posedness of the kernel equations, 
we add the following artificial boundary conditions for $N_{ij}(i<j)$}
\forall 1& \leq i<j \leq m, & N_{ij}(x,0)&=0 \label{eq:artificialboundaryobs}
\end{align}
{while the~$d^+_{ij}$,~$d^-_{ij}$ and $h_{ij}$ are given by}
\begin{align}
	h_{ij}(x)&= N_{ij}(1,x) -\sum\limits_{k=1}^n \rho_{ik} M_{kj}(1,x)
	\end{align}
	\begin{align}
	d^+_{ij}(x,\xi)=-\sum\limits_{k=1}^m M_{ik}(x,\xi)\sigma^{-+}_{kj}
	+\int_\xi^x \sum\limits_{k=1}^m M_{ik}(x,s)d_{kj}^-(s,\xi)ds
		\end{align}
		\begin{align}
	d^-_{ij}(x,\xi)=-\sum\limits_{k=1}^m N_{ik}(x,\xi)\sigma^{-+}_{kj}+\int_\xi^x \sum\limits_{k=1}^m N_{ik}(x,s)d_{kj}^-(s,\xi)ds
\end{align} 
provided the~$M$ and~$N$ kernels are properly defined. Interestingly, the well-posedness of the system of kernel equations of the observer~\eqref{eq:developedKernelM}--\eqref{eq:artificialboundaryobs} is equivalent to that of the controller kernels~\eqref{eq:developedKernelK}--\eqref{eq:artificialboundary}. Indeed, considering the following alternate variables
\begin{align}
			\bar{M}_{ij}(\chi,y)&=M_{ij}(1-y,1-\chi)=M_{ij}(x,\xi),\\
			\bar{N}_{ij}(\chi,y)&=N_{ij}(1-y,1-\chi) = N_{ij}(x,\xi)
	\end{align}
	yields
	
	{\underline{for $1\leq i \leq n$, $1\leq j \leq m$}}
\begin{align}
	\mu_j \partial_\chi \bar{M}_{ij}(\chi,y) - \lambda_i \partial_y \bar{M}_{ij}(\chi,y) =
	-\sum\limits_{k=1}^n\sigma^{++}_{ik}\bar{M}_{kj}(\chi,y)-\sum\limits_{p=1}^m \sigma^{+-}_{ip}\bar{N}_{pj}(\chi,y)
	\end{align}
	{\underline{for $1\leq i \leq m$, $1\leq j \leq m$}}
	\begin{align}
	\mu_j \partial_\chi \bar{N}_{ij}(\chi,y) + \mu_i \partial_y \bar{N}_{ij}(\chi,y) =
	-\sum\limits_{k=1}^n\sigma^{-+}_{ik} \bar{M}_{kj}(\chi,y)-\sum\limits_{p=1}^m \sigma^{--}_{ip}\bar{N}_{pj}(\chi,y)
\end{align}
along with the following set of boundary conditions
\begin{align}
1 \leq i \leq m, \ 1\leq j \leq n \
	\bar{M}_{ij} (\chi,\chi)&= \cfrac{\sigma^{+-}_{ij}}{\mu_i+\lambda_j} \stackrel{\Delta}{=}m_{ij}\\
	\forall 1\leq i,j \leq m, \ i \neq j, \quad \bar{N}_{ij} (\chi,\chi)&= 0\\
	\forall 1\leq j \leq i \leq m, \quad	\bar{N}_{ij}(\chi,0) &=\sum\limits_{k=1}^n \rho_{ik} \bar{M}_{kj}(\chi,0) \\
\forall 1\leq i<j \leq m, \quad \bar{N}_{ij}(1,y)&=0 
\end{align} 
which has the exact same structure as the controller kernel system, the well-posedness of which is assessed in Theorem~\ref{the:wellposednessKandL}.
\subsection{Output feedback controller}
The estimates can be used in an observer-controller scheme to derive an output feedback law yielding finite-time stability of the zero equilibrium. More precisely, we have the following Lemma.
\begin{Lemma}\label{lem:outputfeedback}
	Consider the system composed of the original~\eqref{eq:system1}--\eqref{eq:systemBC} and target systems~\eqref{eq:obs1}--\eqref{eq:obsBC} with the following control law
	\begin{align}
			U(t)&=\int_0^1 \left[K(1,\xi)\hat{u}(\xi) + L(1,\xi)\hat{v}(\xi)\right]d\xi-R_1\hat{u}(t,1)
	\end{align}
	where~$K$ and~$L$ are defined by~\eqref{eq:developedKernelK}--\eqref{eq:artificialboundary}. Its solutions~$(u,v,\hat{u},\hat{v})$ converge in finite time to zero.
\end{Lemma}	
\begin{Proof}
	Proposition~\ref{prop:stabtargetObs} along with the existence of the observer backstepping transformation~\eqref{eq:BS1},\eqref{eq:BS2} yields convergence of the observer error states~$\tilde{u}$, $\tilde{v}$ defined by~\eqref{eq:errorobs1}--\eqref{eq:errorobsBC} to zero for~$t\geq t_F$\footnote{the proof of this claim follows the exact same steps as in the controller case, see Section~\ref{sec:main}}. Therefore, for~$t\geq t_F$, one has~$v(t,0)=\hat{v}(t,0)$ and Theorem~\ref{the:main} applies to the observer system~\eqref{eq:obs1}--\eqref{eq:obsBC}. Therefore, for~$t\geq 2t_F$, one has~$(\tilde{u},\tilde{v},\hat{u},\hat{v})\equiv 0$ which also yields~$(u,v)\equiv 0$.
\end{Proof}

\section{Motion planning for homodirectional systems}\label{sec:motion_planning}
\subsection{Definition of the motion planning problem}
Consider now the case $n=0$. Then system (\ref{eq:system1})--(\ref{eq:systemBC}) reduces to
\begin{align}
	v_t (t,x) - \Lambda^- v_x (t,x) &= \Sigma^{--} v(t,x),\label{eq:system2homo}
\end{align}
where coefficients $ \Lambda^-$ and $\Sigma^{--}$ defined as in (\ref{eq:coefLambda}) and (\ref{eq:coefSigma}), with boundary conditions
\begin{align}
v(t,1)&= U(t) \label{eq:systemBChomo}.
\end{align}
For simplicity in this section we drop the super-indices in the coefficients.

Equation (\ref{eq:system2homo}) represents a system of $m$ states moving in the same direction (in this case, from right to left). We call such a system \emph{homodirectional} (in oposition with \emph{heterodirectional} systems, whose states move in different directions, such as (\ref{eq:system1})--(\ref{eq:systemBC}) with $n,m\neq0$). Homodirectional systems are inherently finite-time stable. Physically, this is due to the fact that they are transport equations with information  flowing only in one direction; thus, setting $U(t)$ to zero in (\ref{eq:systemBChomo}) and solving the equations with the method of characteristics, we obtain $u(t,x)\equiv0$ for $t\geq\frac{1}{\mu_m}$ (the slowest transport time in (\ref{eq:system2homo})).

For (\ref{eq:system2homo})--(\ref{eq:systemBChomo}) we consider the following \emph{motion planning} problem. Given $\Phi(t)$, a known function defined as
\begin{align}
	\Phi(t)&=\begin{pmatrix}
		\Phi_1(t)&\cdots&\Phi_n(t)
	\end{pmatrix}^T,
\end{align}
find the value of $U(t)$ so that $v(t,0)\equiv\Phi(t)$ for $t\geq t_M$, for some $t_M>0$.
\begin{rmk}
Even though the plant  (\ref{eq:system2homo}) is finite-time stable, and a formula for the states can be written by using the method of characteristics, the motion planning problem is not trivial to solve. The entanglement of different states moving with different speeds severely complicates finding a solution. This design difficulty will be explicitly shown with an example in Section~\ref{sec:example}.
\end{rmk}
\subsection{Tracking control design}
The following result solves the motion planning problem.
\begin{theorem}\label{th:motionplanning}
Consider system (\ref{eq:system2homo}) with boundary conditions (\ref{eq:systemBChomo}), initial condition~$v_0\in (L^2(0,1))^{m}$, and feedback control law
\begin{align}\label{feedback law homo}
 U_i(t) =\Phi_i\left(t+\frac{1}{\mu_i}\right)+ \sum_{j=1}^{i=m} \int_0^1 L_{ij}(1,\xi)v_j(\xi)d\xi
-\sum_{j=1}^{i-1} \int_0^1 \frac{\mu_j}{\mu_i} L_{ij}(\xi,0) \Phi_j\left(t+\frac{1-\xi}{\mu_i}\right) d\xi
\end{align}
Then, $v(t,0)\equiv\Phi(t)$ if $t\geq t_M$, for $t_M=\sum_{j=1}^m \frac{1}{\mu_j}$.
\end{theorem}
\begin{rmk}
The motion planning problem has been solved for the homodirectional case for the sake of clarity. However, it can be formulated for the full heterodirectional system~\eqref{eq:system1}--\eqref{eq:systemBC} with only minor modifications. Noting $u(t,0)= Q_0 v(t,0)$, the values of some $u_i$'s could be chosen as part of the output instead of some of the $v_i$'s,  for a total of $m$ states. The only condition would be that all the rows of the output vector (written in terms of the $v_i$'s) are linearly independent.
\end{rmk}
\begin{proof}
We start by using the backstepping transformation (\ref{eq:BS})---where the kernels $K$ are zero due to $n$ being zero---to map (\ref{eq:system2homo})--(\ref{eq:systemBChomo}) into the target system
\begin{align}
	\beta_t (t,x) - \Lambda \beta_x (t,x) &=   G(x) \beta(0),\label{eq:target2homo}
\end{align}
where $G(x)$ was defined in Section~\ref{sec:transformation} as a function of the kernels,
with the following boundary conditions
\begin{align}
\beta(t,1)&= B(t),\label{eq:targetBChomo}
\end{align}
where $B(t)$ in (\ref{eq:targetBChomo}) is a function defined as
\begin{align}
	B(t)&=\begin{pmatrix}
		B_1(t)&\cdots&B_n(t)
	\end{pmatrix}^T,
\end{align}
with components to be determined. $B$ represents an extra degree of freedom that did not appear in the target system for the homodirectional control problem (Equation~\ref{eq:targetBC}). It will be used to solve the motion planning problem. The presence of $B(t)$ in the boundary conditions does not change the backstepping transformation; however it modifies the feedback control law to
\begin{align}\label{eq:controlmotion}
 U(t) = B(t)+\int_0^1 L(1,\xi)v(\xi)d\xi.
\end{align}
Now, noticing that if one sets $x=0$ in the transformation (\ref{eq:BS}) one obtains $v_i(t,0)=\beta_i(t,0)$, it is clear that we only need to solve the motion planning problem for the target $\beta$ system by using $B(t)$. The next steps of the proof are devoted to finding the value of $B(t)$. 

Using the method of characteristics, the explicit solution for each state $\beta_i(t,x)$ of (\ref{eq:targetBChomo}) with boundary condition (\ref{eq:targetBChomo}) at time $t\geq\frac{1-x}{\mu_i}$ is
\begin{align}
\beta_i(t,x)=B_i\left(t+\frac{x-1}{\mu_i}\right)
+\frac{1}{\mu_i}\int_x^1 G(\xi) \beta\left(t+\frac{x-\xi}{\mu_i},0\right) d\xi, \label{eq:betai}
\end{align}
Using (\ref{matrixG})
and (\ref{eq:gijfunctionLij}) in (\ref{eq:betai}), we obtain
\begin{align}
\beta_i(t,x)=B_i\left(t+\frac{x-1}{\mu_i}\right)+\sum_{j=1}^{i-1} \int_x^1 \frac{\mu_j}{\mu_i} L_{ij}(\xi,0) \beta_j\left(t+\frac{x-\xi}{\mu_i},0\right) d\xi.\label{eq:betai3}
\end{align}

To solve now the motion planning problem, consider first (\ref{eq:betai3}) for $i=1$ and $x=0$, for $t\geq\frac{1}{\mu_1}$. Imposing $\beta_1(t,0)=\Phi_1(t)$, we obtain:
\begin{align}
\Phi_1(t)=B_1\left(t-\frac{1}{\mu_1}\right),
\end{align}
thus, setting
$B_1(t)=\Phi_1 \left(t+\frac{1}{\mu_1}\right)$ for $t\geq0$,
we obtain the desired behavior for $\beta_1(t,0)$ for $t\geq\frac{1}{\mu_1}$.
Now consider (\ref{eq:betai3}) for $i=2$ and $x=0$, for $t\geq\frac{1}{\mu_2}$. Imposing $\beta_2(t,0)=\Phi_2(t)$, we obtain:
\begin{align}
\Phi_2(t)=B_2\left(t-\frac{1}{\mu_2}\right)+ \int_0^1 \frac{\mu_2}{\mu_1} L_{21}(\xi,0) \beta_1\left(t-\frac{\xi}{\mu_2},0\right) d\xi.\label{eq:beta2}
\end{align}
Solving for $B_2$ as before
\begin{align}\label{eq:beta3}
B_2\left(t\right)=\Phi_2\left(t+\frac{1}{\mu_2}\right)- \int_0^1 \frac{\mu_1}{\mu_2} L_{21}(\xi,0)  \beta_1\left(t+\frac{1-\xi}{\mu_2},0\right) d\xi.
\end{align}
To be able to substitute $\beta_1(t,0)$ for $\Phi_1(t)$ in the whole domain of the integral in (\ref{eq:beta3}) we need to wait until $t=\frac{1}{\mu_1}$. Thus choosing
\begin{align}\label{eq:beta4}
B_2\left(t\right)=\Phi_2\left(t+\frac{1}{\mu_2}\right)- \int_0^1 \frac{\mu_1}{\mu_2} L_{21}(\xi,0)  \Phi_1\left(t+\frac{1-\xi}{\mu_2},0\right) d\xi,
\end{align}
we get that $\beta_2(t,0)=\Phi_2(t)$ for $t\geq\frac{1}{\mu_1}+\frac{1}{\mu_2}$ (as we have to wait an extra $\frac{1}{\mu_2}$ time for (\ref{eq:beta4}) to propagate).
It is clear that this procedure can be continued for $i=3,\hdots,m$. Thus we obtain that
\begin{align}
B_i\left(t\right)=\Phi_i\left(t+\frac{1}{\mu_i}\right)-\sum_{j=1}^{i-1} \int_0^1 \frac{\mu_j}{\mu_i} L_{ij}(\xi,0) \Phi_j\left(t+\frac{1-\xi}{\mu_i}\right) d\xi
\label{eqn:Bivalues}
\end{align}
solves the motion problem for $\beta_i$ for $t\geq \sum_{j=1}^i \frac{1}{\mu_j}$. Applying (\ref{eqn:Bivalues}) for $i=1,\hdots,m$ and substituting in (\ref{eq:controlmotion}) produces the feedback law (\ref{feedback law homo}), thus solving the motion planning problem in time $t_M=\sum_{j=1}^m \frac{1}{\mu_j}$.
\end{proof}
\begin{rmk}
Theorem~\ref{th:motionplanning} gives in fact \emph{tracking} (in finite-time) of the desired output signal, a result stronger than pure motion planning. To obtain a pure motion planning result, one should take (\ref{eq:betai3})---the explicit solutions of the target system obtained in the proof of the theorem---and substitute the values of $B_i$ found in (\ref{eqn:Bivalues}), so that the $\beta_i$'s are explicit functions of the $\Phi_i$'s. Then, using the inverse backstepping transformation (\ref{invertible}), find the $v_i$'s as explicit functions of the $\Phi_i$'s and substitute them in the control law (\ref{feedback law homo}), which would then be an exclusive function of the outputs. We omit this result for lack of space.
\end{rmk}
\subsection{An explicit motion planning example}\label{sec:example}
Next we present an specific example of a motion planning problem for $m=2$. Consider the plant
\begin{align}
	v_{1t} (t,x) - \mu_1 v_{1x} (t,x) &= \sigma_{12} v_2(t,x),\label{equ:v1}\\
	v_{2t} (t,x) - \mu_2 v_{2x} (t,x) &= \sigma_{21} v_1(t,x),\label{equ:v2}
\end{align}
with boundary conditions
\begin{align}
v_1(t,1)&= U_1(t), \quad v_2(t,1)=U_2.\label{equ:bcv1v2}
\end{align}
The objective is to design $U_1(t)$ and $U_2(t)$ so that $v_1(t,0)=\Phi_1(t)$ and $v_2(t,0)=\Phi_2(t)$ for some functions $\Phi_1,\Phi_2$ for $t\geq t_M$. Notice that since (\ref{equ:v1})--(\ref{equ:bcv1v2}) is explicitly solvable, one might think that the inputs can be directly designed. Using the method of characteristics to explicitly write a solution of the system, one gets, after time $t=\frac{1}{\mu_2}$,
\begin{align}\label{eqn-utphi1}
v_1(t,0)&=
U_1\left(t-\frac{1}{\mu_1}\right) 
+\frac{1}{\mu_1}
\int_0^{1}\sigma_{12} v_2 \left(t-\frac{\xi}{\mu_1},\xi \right)d\xi, \\
v_2(t,0)&=
U_2\left(t-\frac{1}{\mu_2}\right) 
+\frac{1}{\mu_2}
\int_0^{1}\sigma_{21} v_1 \left(t-\frac{\xi}{\mu_2},\xi \right)d\xi. \label{eqn-utphi2}
\end{align}
However, if one tries to proceed as in the proof of Theorem~\ref{th:motionplanning}, by plugging in $\Phi_1(t)$ in (\ref{eqn-utphi1}) and $\Phi_2(t)$ in (\ref{eqn-utphi2}), and then solve for $U_1(t)$ and $U_2(t)$, one ends up with a feedback law that requires knowing \emph{future} values of $v_1$ and $v_2$, i.e., a \emph{non-causal} (and therefore not implementable) feedback law. Thus, a direct approach does not work even for the $m=2$ case. To solve the motion planning problem, we resort to Theorem~\ref{th:motionplanning}; in this particular case, the motion planning problem is solved by the inputs
\begin{align}\label{feedback law homo example}
 U_1(t) =\Phi_1\left(t+\frac{1}{\mu_1}\right)+ \int_0^1 L_{11}(1,\xi)v_1(\xi)d\xi+ \int_0^1 L_{12}(1,\xi)v_2(\xi)d\xi,
\end{align}
\begin{align}
 U_2(t) =\Phi_2\left(t+\frac{1}{\mu_2}\right)
- \int_0^1 \frac{\mu_1}{\mu_2} L_{21}(\xi,0) \Phi_1\left(t+\frac{1-\xi}{\mu_2}\right) d\xi+  \int_0^1 L_{21}(1,\xi)v_1(\xi)d\xi+ \int_0^1 L_{22}(1,\xi)v_2(\xi)d\xi,
\end{align}
where the kernels $L_{11}$, $L_{12}$, $L_{21}$ and $L_{22}$ satisfy
\begin{align} \label{eqn-ku1}
\mu_1 \partial_x L_{11}(x,\xi)+\mu_1 \partial_\xi L_{11}(x,\xi)&=
\sigma_{21}  L_{12}(x,\xi)\\
\mu_1 \partial_x L_{12}(x,\xi)+\mu_2 \partial_\xi L_{12}(x,\xi)&=\sigma_{12} L_{11}(x,\xi) ,\\
\mu_2 \partial_x L_{21}(x,\xi)+\mu_1 \partial_\xi L_{21}(x,\xi)&=
\sigma_{21} L_{22}(x,\xi)
\\
\mu_2 \partial_x L_{22}(x,\xi)+\mu_2 \partial_\xi L_{22}(x,\xi)&=\sigma_{12} L_{21}(x,\xi)
,
\end{align}
with boundary conditions
\begin{align}
 L_{11}(x,0)&= L_{12}(x,0)=L_{22}(x,0)=0,\\
L_{12}(x,x) &=\frac{\sigma_{12}}{\mu_2-\mu_1},\quad
L_{21}(x,x)=\frac{\sigma_{21}}{\mu_1-\mu_2}, \label{eq:bcl21}
\end{align}
plus the artificial boundary condition $L_{21}(1,\xi)=l_{21}(\xi)$, where the function $l_{21}$ is arbitrary. These kernel PDEs can be explicitly solved using techniques akin to those used in~\cite{Vazquez2014}. The resulting kernels (whose validity can be verified by substitution in the kernel equations) are given by~\eqref{eq:longkernels1}--\eqref{eq:longkernelsend}.
\begin{figure*}
\hrule
\footnotesize
\begin{align}
L_{11}(x,\xi)=& \begin{cases}
	 \dfrac{\sqrt{\sigma_{12}\sigma_{21}}}{\mu_2-\mu_1}
\sqrt{\dfrac{\mu_1 \xi -\mu_2 x}{\mu_1(x - \xi)}}
\mathrm{I}_1 \left( 
\dfrac{2}{\mu_1 - \mu_2}\sqrt{\dfrac{\sigma_{12}\sigma_{21}(x - \xi)(\mu_1 \xi - \mu_2 x)}{\mu_1}}
\right),& \xi\geq \frac{\mu_2}{\mu_1}x \\ 0, & \xi< \frac{\mu_2}{\mu_1}x 
\end{cases} \label{eq:longkernels1}\\
L_{12}(x,\xi)=& \begin{cases}
	\dfrac{\sigma_{21}}{\mu_2-\mu_1}
\mathrm{I}_0\left( 
\dfrac{2}{\mu_1 - \mu_2}\sqrt{\dfrac{\sigma_{12}\sigma_{21}(x - \xi)(\mu_1 \xi - \mu_2 x)}{\mu_1}}
\right),& \xi\geq \frac{\mu_2}{\mu_1}x \\ 0, & \xi< \frac{\mu_2}{\mu_1}x 
\end{cases}\\
L_{21}(x,\xi)=&\frac{\sigma_{21}\xi}{\mu_1 x-\mu_2 \xi} \mathrm{J}_0 \left( \frac{2}{\mu_1-\mu_2} \sqrt{\frac{\sigma_{12}\sigma_{21}(x-\xi)(\mu_1 x-\mu_2\xi)}{\mu_2}}\right)+\mu_1\sqrt{\frac{\sigma_{21} \mu_2 (x-\xi)}{\sigma_{12} (\mu_1 x - \mu_2 \xi)^3}}
\mathrm{J}_1 \left( \frac{2}{\mu_1-\mu_2} \sqrt{\frac{\sigma_{12}\sigma_{21}(x-\xi)(\mu_1 x-\mu_2\xi)}{\mu_2}}\right)
,\\
L_{22}(x,\xi)=&\xi \sqrt{\frac{\sigma_{12}\sigma_{21}}{\mu_2(x-\xi)(\mu_1 x-\mu_2\xi)}}\mathrm{J}_1 \left( \frac{2}{\mu_1-\mu_2} \sqrt{\frac{\sigma_{12}\sigma_{21}(x-\xi)(\mu_1 x-\mu_2\xi)}{\mu_2}}\right) \label{eq:longkernelsend}
\end{align}
\hrule
\end{figure*}
where $\mathrm{I}_0$ and $\mathrm{I}_1$ are the modified Bessel functions of order 0 and 1, and $\mathrm{J}_0$ and $\mathrm{J}_1$ are the (regular) Bessel functions of order 0 and 1, respectively.

The kernels appearing in (\ref{eqn-utphi1})--(\ref{eqn-utphi2}) are depicted in Fig~\ref{fig:kernels} for the case $\mu_1=1$, $\mu_2=0.2$ and $\sigma_{12}=2$, $\sigma_{21}=5$. It can be seen that $L_{11}(1,\xi)$ and $L_{12}(1,\xi)$ have a monotone behaviour (they are always negative or zero), whereas $L_{21}(1,\xi)$, $L_{21}(\xi,0)$, and $L_{22}(1,\xi)$ are oscillatory. Fig.~\ref{fig:l12} shows $L_{11}$ and $L_{12}$ in the whole domain $\mathcal{T}$; notice that $L_{12}(x,\xi)$ is discontinuous along the line $\xi=\frac{\mu_2}{\mu_1}$ (which is the lower domain on Figure~\ref{fig:char3}), whereas $L_{11}(x,\xi)$ is not discontinuous. On the other hand, it is evident that $l_{21}(\xi)=L_{21}(1,\xi)$ is rather non-trivial. In fact, the procedure that was followed to find these explicit solutions was not setting a value of $l_{21}$ a priori, but rather \emph{extending} the domain shown in Figure~\ref{fig:char2} up to $x=\frac{\mu_1}{\mu_1-\mu_2}$, so that boundary condition (\ref{eq:bcl21}) can be used to actually find the value of $l_{21}$.
\begin{figure}[h]%
\centering
\includegraphics[width=.8\columnwidth]{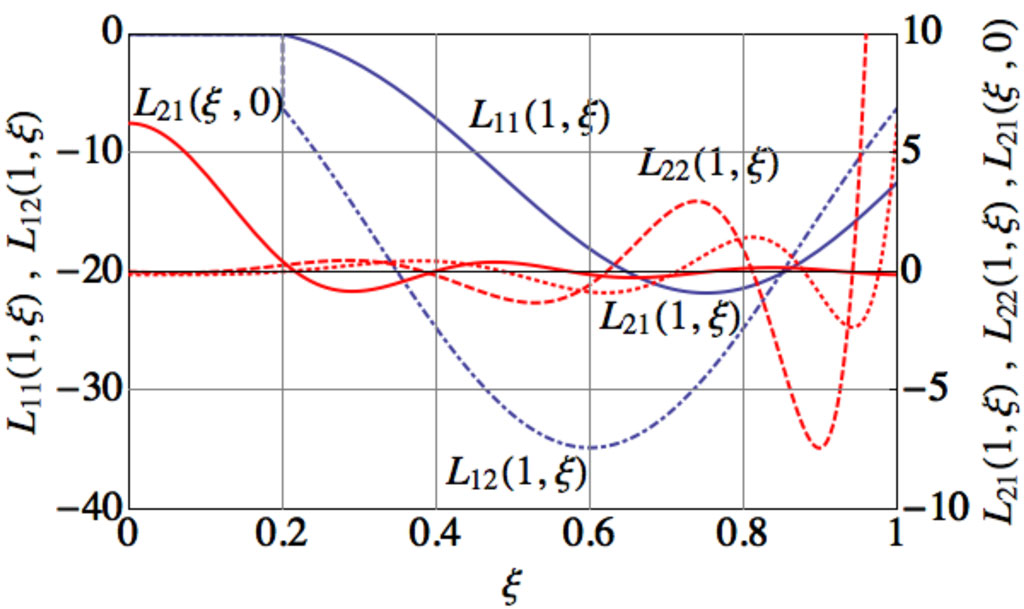}%
\caption{Motion planning kernels ($n=0$, $m=2$). Solid: $L_{11}(1,\xi)$ and $L_{21}(\xi,0)$. Dash-dotted: $L_{12}(1,\xi)$. Dotted: $L_{21}(1,\xi)$. Dashed: $L_{22}(1,\xi)$.}%
\label{fig:kernels}%
\end{figure}

\begin{figure}[h]%
\centering
\begin{subfigure}{.45\columnwidth}%
\includegraphics[width=\columnwidth]{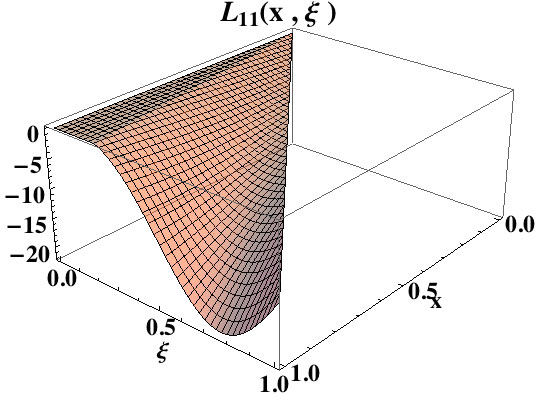}%
\end{subfigure}
\begin{subfigure}{.45\columnwidth}%
\centering
\includegraphics[width=\columnwidth]{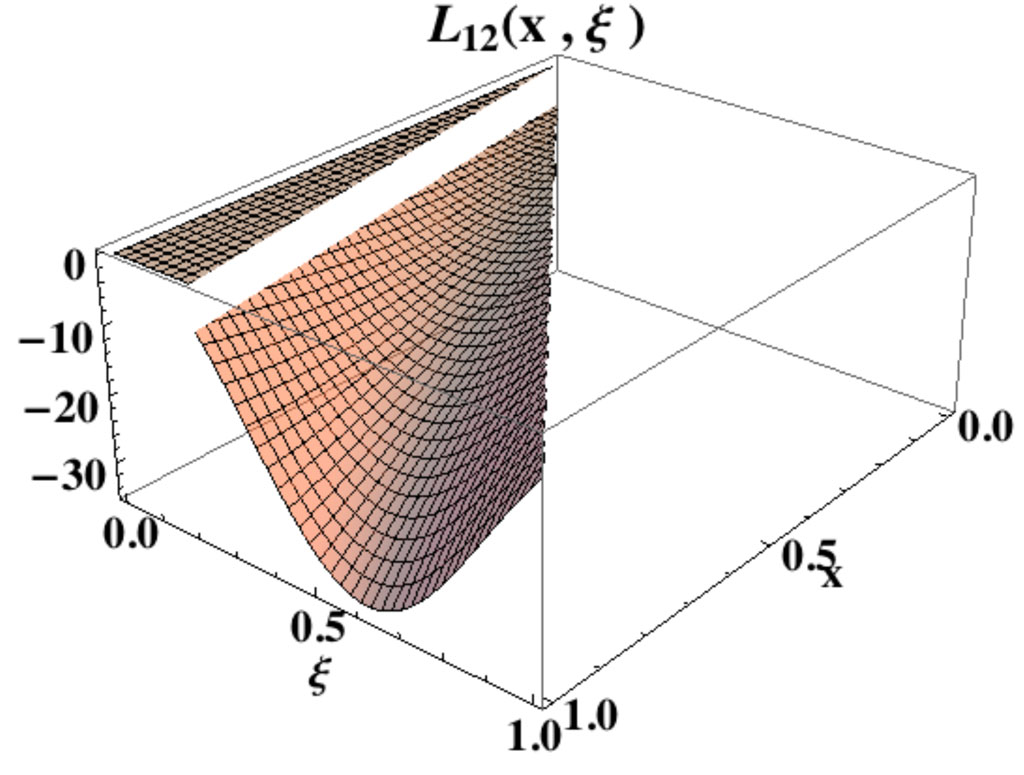}%
\end{subfigure}
\caption{Motion planning kernels $L_{11}(x,\xi)$ and $L_{12}(x,\xi)$ ($n=0$, $m=2$).}%
\label{fig:l12}%
\end{figure}
\section{Proof of Theorem~\ref{the:wellposednessKandL}: well-posedness of the kernel equations}\label{sec:wellposedness}
To prove well-posedness of the kernel equations, we classically transform them into integral equations and use the method of successive approximations.
\begin{rmk}
	Similar proofs have been derived for less general systems, e.g. in~\cite{Coron2013} or~\cite{DiMeglio2013}. The proof is more involved here due to the existence of \emph{homodirectional controlled} states, which lead to the homodirectional kernel PDEs~\eqref{eq:developedKernelL}.
\end{rmk} 
\subsection{Method of characteristics}
\subsubsection{Characteristics of the~$K$ kernels}
For each~$1\leq i \leq m,$ $1\leq j\leq n$, and~$(x,\xi) \in \mathcal{T}$, we define the following characteristic lines~$(x_{ij}(x,\xi;\cdot),\xi_{ij}(x,\xi;\cdot))$ corresponding to Equations~\eqref{eq:developedKernelK}
\begin{align}
	&\left\{\begin{aligned}
		\frac{dx_{ij}}{ds}(x,\xi;s)&=-\mu_i, \quad s\in \left[0,s_{ij}^F(x,\xi)\right]\\
		x_{ij}(x,\xi;0)&=x,\ x_{ij}(x,\xi;s_{ij}^F(x,\xi))=x_{ij}^F(x,\xi)
	\end{aligned}\right.,\label{eq:charK1}
	\end{align}
	\begin{align}
	&\left\{\begin{aligned}
		\frac{d\xi_{ij}}{ds}(x,\xi;s)&=\lambda_j, \quad s\in \left[0,s_{ij}^F(x,\xi)\right]\\
		\xi_{ij}(x,\xi;0)&=\xi,\ \xi_{ij}(x,\xi;s_{ij}^F(x,\xi))=x_{ij}^F(x,\xi)
	\end{aligned}\right.\label{eq:charK2}
\end{align}
These lines, depicted on Figure~\ref{fig:char1}, originate at the point~$(x,\xi)$ and terminate on the hypothenuse at the point~$\left(x^F_{ij}(x,\xi),x^F_{ij}(x,\xi)\right)$. The expressions of~$x_{ij}(x,\xi;s)$, $\xi_{ij}(x,\xi;s)$ $s_{ij}^F(x,\xi)$ and~$x_{ij}^F(x,\xi)$ are omitted here because of lack of space, but are straightforward to compute. 
\begin{figure}
\centering
\includegraphics[width=.75\columnwidth]{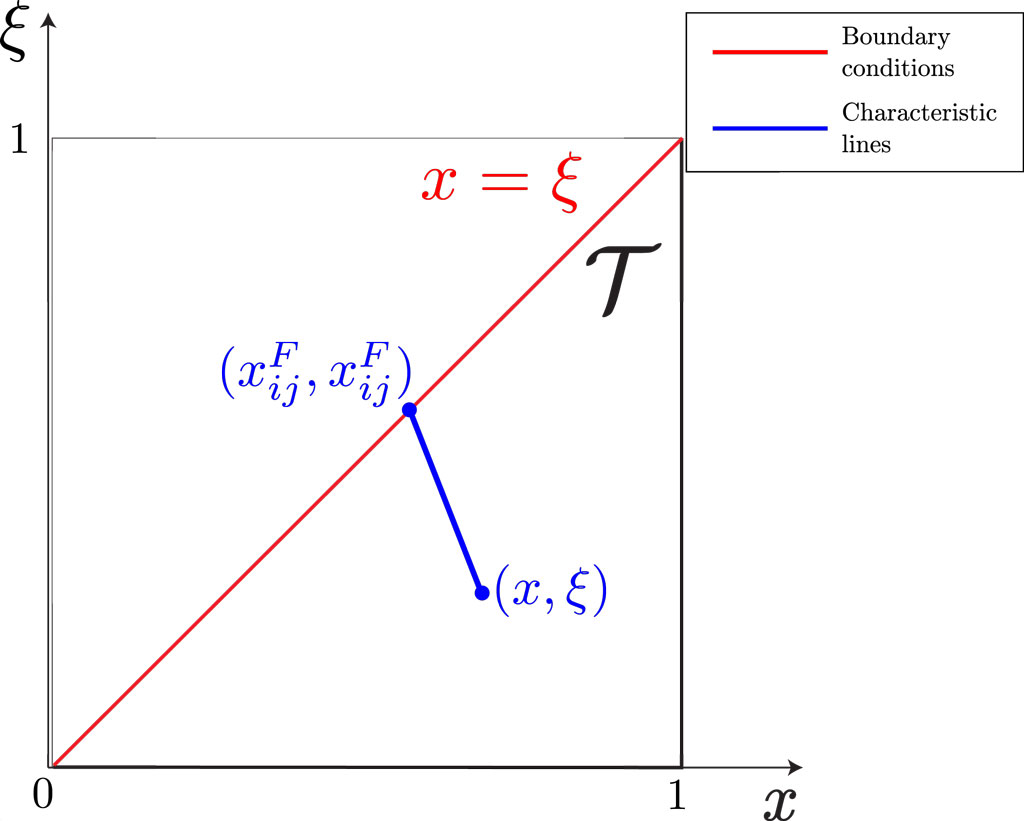}%
\caption{Characteristic lines of the $K$ kernels}%
\label{fig:char1}%
\end{figure}
Integrating~\eqref{eq:developedKernelK} along these characteristic lines and plugging in the boundary condition~\eqref{eq:hypotenuseK} yields
\begin{align}
	K_{ij}(x,\xi)=k_{ij}
	+\int_0^{s_{ij}^F(x,\xi)} \left[\sum\limits_{k=1}^n\sigma^{++}_{kj}K_{ik}(x_{ij}(x,\xi;s),\xi_{ij}(x,\xi;s))\right.
	\left.+\sum\limits_{p=1}^m \sigma^{-+}_{pj}L_{ip}(x_{ij}(x,\xi;s),\xi_{ij}(x,\xi;s))\right] ds\label{eq:integralK}
\end{align}
\subsubsection{Characteristics of the~$L$ kernels}
For each~$1\leq i \leq m,$ $1\leq j\leq m$, and~$(x,\xi) \in \mathcal{T}$, we define the following characteristic lines~$(\chi_{ij}(x,\xi;\cdot),\zeta_{ij}(x,\xi;\cdot))$ corresponding to Equations~\eqref{eq:developedKernelL}
\begin{align}
	&\left\{\begin{aligned}
		\frac{d\chi_{ij}}{d\nu}(x,\xi;\nu)&=\epsilon_{ij}\mu_i, \quad \nu\in \left[0,\nu_{ij}^F(x,\xi)\right]\\
		\chi_{ij}(x,\xi;0)&=x,\ \chi_{ij}(x,\xi;\nu_{ij}^F(x,\xi))=\chi_{ij}^F(x,\xi)
	\end{aligned}\right.,\label{eq:charL1}
	\end{align}
	\begin{align}
	&\left\{\begin{aligned}
		\frac{d\zeta_{ij}}{d\nu}(x,\xi;\nu)&=\epsilon_{ij}\mu_j, \quad \nu\in \left[0,\nu_{ij}^F(x,\xi)\right]\\
		\zeta_{ij}(x,\xi;0)&=\xi,\ \zeta_{ij}(x,\xi;\nu_{ij}^F(x,\xi))=\zeta_{ij}^F(x,\xi)
	\end{aligned}\right.\label{eq:charL2}
\end{align}
where~$\epsilon_{ij}$ is defined by
\begin{align}
	\epsilon_{ij}(x,\xi)&=\begin{cases}
		1&\text{if }i>j\\
		-1&\text{otherwise}
	\end{cases}\label{eq:epsilonij}
\end{align}
These lines all originate at~$(x,\xi)$ and terminate on~$\partial \mathcal{T}$ at the point~$\left(\chi_{ij}^F(x,\xi),\zeta_{ij}^F(x,\xi)\right)$. They are depicted on Figures~\ref{fig:char2}--\ref{fig:char3} in the three distinct cases~$i<j$,~$i=j$ and~$i>j$. The detailed expressions of~$\chi_{ij}(x,\xi;s)$, $\zeta_{ij}(x,\xi;s)$ $\nu_{ij}^F(x,\xi)$,~$\chi_{ij}^F(x,\xi)$ and~$\zeta_{ij}^F(x,\xi)$ are, again, omitted here because of space constraints.
\begin{figure}
\centering
\includegraphics[width=.75\columnwidth]{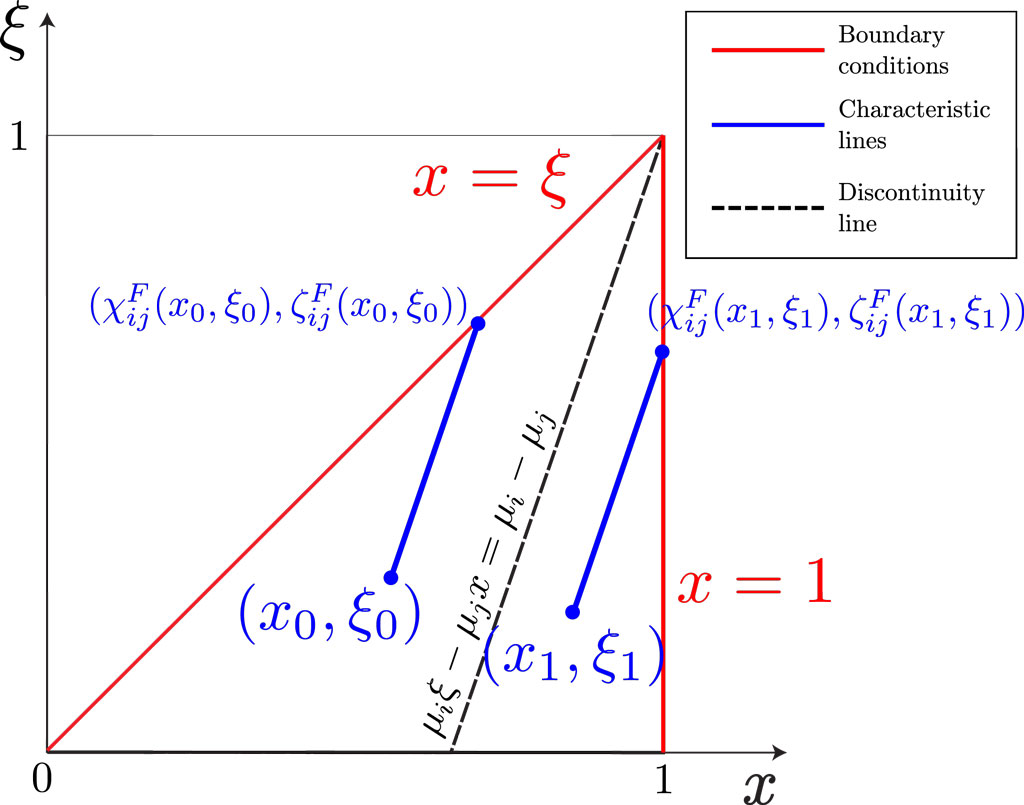}%
\caption{Characteristic lines of the kernels~$L_{ij}$ for~$i>j$}%
\label{fig:char2}%
\end{figure}
Integrating~\eqref{eq:developedKernelL} along these characteristics and plugging in the boundary conditions~\eqref{eq:hypotenuseL},\eqref{eq:modifx0boundary} and \eqref{eq:artificialboundary} yields
\begin{multline}
	L_{ij}\left(x,\xi\right)=\delta_{ij}(x,\xi) l_{ij}+ \left(1-\delta_{ij}(x,\xi)\right) \frac{1}{\mu_j}\sum\limits_{r=1}^n \lambda_r q_{rj}  K_{ir}(\chi_{ij}^F(x,\xi),0)\\
		-\epsilon_{ij}\int_0^{\nu_{ij}^F(x,\xi)} \left[\sum\limits_{p=1}^m\sigma^{--}_{pj}L_{ip}\left(\chi_{ij}(x,\xi;\nu),\zeta_{ij}(x,\xi;\nu)\right)\right.
		\left.+\sum\limits_{k=1}^n \sigma^{+-}_{kj}K_{ik}\left(\chi_{ij}(x,\xi;\nu),\zeta_{ij}(x,\xi;\nu)\right)\right] d\nu
\end{multline}
where the coefficient~$\delta_{ij}(x,\xi)$, defined by
\begin{align}
	\delta_{ij}(x,\xi)&= \begin{cases}
		0 & \text{if $i=j$}\\
		0 & \text{if $i<j$ and $\mu_i \xi-\mu_j x \leq 0$}\\
		1 & \text{otherwise}
	\end{cases},
	\end{align}
reflects the fact that some characteristics terminate on the~$\xi=0$ boundary of~$\mathcal{T}$, while others terminate on the hypotenuse or on the $x=1$ boundary of~$\mathcal{T}$. Plugging in~\eqref{eq:integralK} evaluated at~$(\chi_{ij}^F(x,\xi),0)$ yields
\begin{multline}
	L_{ij}\left(x,\xi\right)=\delta_{ij}(x,\xi) l_{ij}+ \left(1-\delta_{ij}(x,\xi)\right) \frac{1}{\mu_j}\sum\limits_{r=1}^n \lambda_r q_{rj}  k_{ir}\\
	+ \left(1-\delta_{ij}(x,\xi)\right) \frac{1}{\mu_j}\sum\limits_{r=1}^n \lambda_r q_{rj}\int_0^{s_{ir}^F(\chi_{ij}^F(x,\xi),0)} \left[\sum\limits_{k=1}^n\sigma^{++}_{kr}K_{ik}(x_{ir}(\chi_{ij}^F(x,\xi),0;s),\xi_{ir}(\chi_{ij}^F(x,\xi),0;s))\right.\\
	\left.+\sum\limits_{p=1}^m \sigma^{-+}_{pr}L_{ip}(x_{ir}(\chi_{ij}^F(x,\xi),0;s),\xi_{ir}(\chi_{ij}^F(x,\xi),0;s)) \right]ds\\
		-\epsilon_{ij}\int_0^{\nu_{ij}^F(x,\xi)} \left[\sum\limits_{p=1}^m\sigma^{--}_{pj}L_{ip}\left(\chi_{ij}(x,\xi;\nu),\zeta_{ij}(x,\xi;\nu)\right)+\right.
		\left.\sum\limits_{k=1}^n \sigma^{+-}_{kj}K_{ik}\left(\chi_{ij}(x,\xi;\nu),\zeta_{ij}(x,\xi;\nu)\right)\right] d\nu \label{eq:integralL}
\end{multline}
\begin{figure}
\centering
\includegraphics[width=.75\columnwidth]{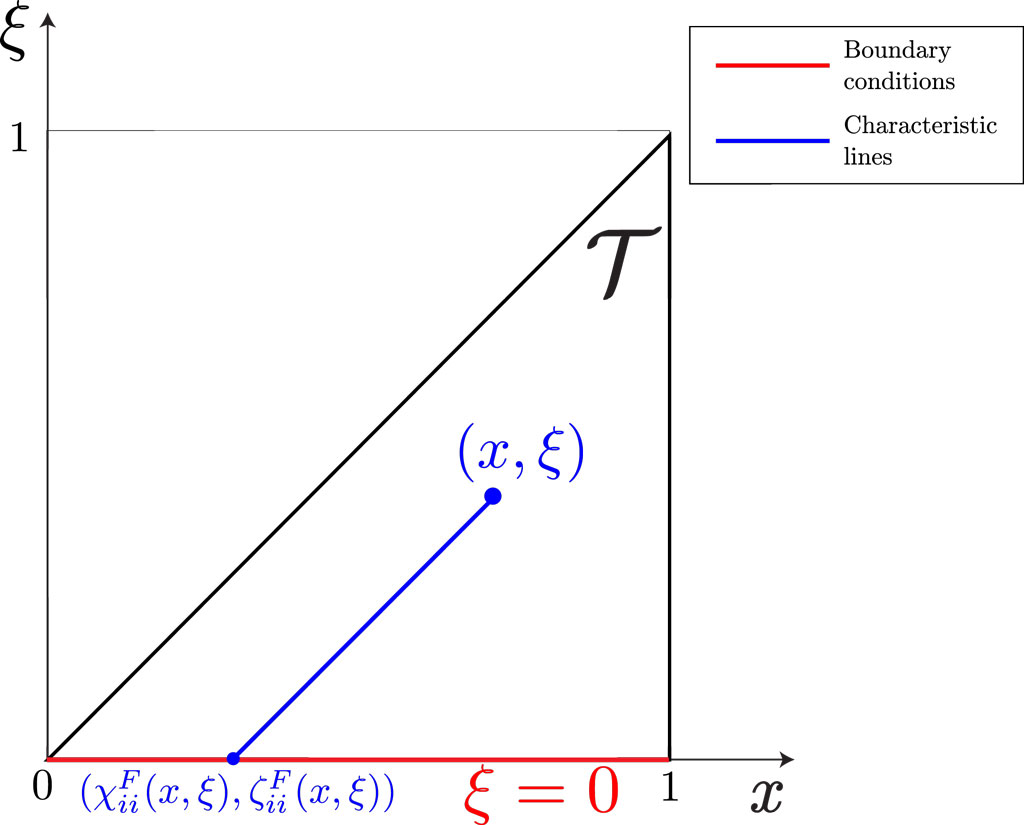}%
\caption{Characteristic lines of the kernels~$L_{ii}$}%
\label{fig:char4}%
\end{figure}
\begin{figure}
\centering
\includegraphics[width=.75\columnwidth]{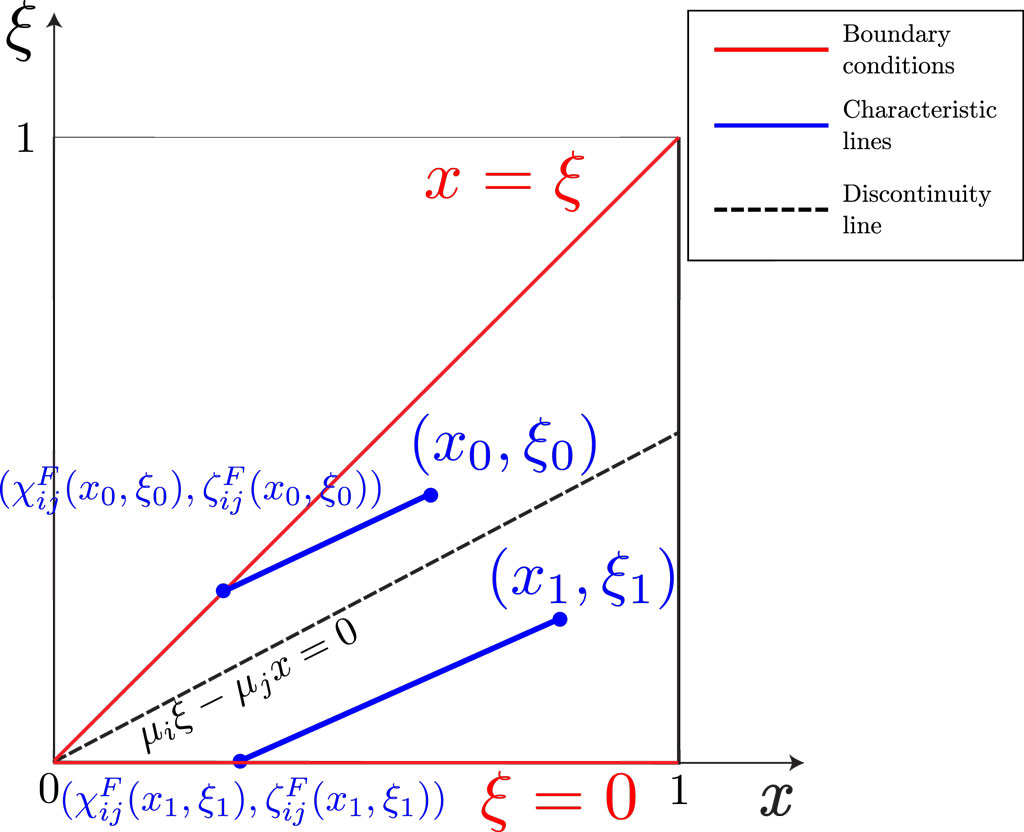}%
\caption{Characteristic lines of the kernels~$L_{ij}$ for~$i<j$}%
\label{fig:char3}%
\end{figure}
\subsection{Method of successive approximations}
We now use the method of successive approximations to solve equations~\eqref{eq:integralK},\eqref{eq:integralL}. Define first
\begin{align}
	 &\forall {1\leq i\leq m, 1\leq j\leq n}, \quad \varphi_{ij}(x,\xi)= k_{ij},\\
\notag&\forall{1\leq i\leq m,1\leq j\leq m},\\
	&\psi_{ij}(x,\xi) = \delta_{ij}(x,\xi) l_{ij}+ \left(1-\delta_{ij}(x,\xi)\right) \frac{1}{\mu_j}\sum\limits_{r=1}^n \lambda_r q_{rj}  k_{ir}
\end{align}
Besides, we define~$\mathbf{H}$ as the vector containing all the kernels, reordered line by line and stacked up, and similarly~$\boldsymbol{\phi}$, as follows
\begin{align}
	\mathbf{H}&=\begin{pmatrix}
		H_1\\ \vdots\\ H_{nm}  \\ H_{nm+1} \\ \vdots \\H_{nm+m^2}
	\end{pmatrix}= \begin{pmatrix}
		K_{11}\\ \vdots \\  K_{mn} \\ L_{11} 
		\\ \vdots \\L_{mm}
	\end{pmatrix},&\boldsymbol{\phi}&=\begin{pmatrix}
		\phi_1\\ \vdots \\ \phi_{nm}  \\ \phi_{nm+1} \\ \vdots \\\phi_{nm+m^2}
	\end{pmatrix}= \begin{pmatrix}
		\varphi_{11}\\ \vdots \\\varphi_{mn} \\ \psi_{11} 
		\\ \vdots \\\psi_{mm}
	\end{pmatrix}
\end{align}
We consider the following linear operators acting on~$\mathbf{H}$, for~$1\leq i\leq m, 1\leq j\leq n$
\begin{align}
	\Phi_{ij}[\mathbf{H}](x,\xi)= \int_0^{s_{ij}^F(x,\xi)} \left[\sum\limits_{k=1}^n\sigma^{++}_{kj}K_{ik}(x_{ij}(x,\xi;s),\xi_{ij}(x,\xi;s))\right.
	\left.+\sum\limits_{p=1}^m \sigma^{-+}_{pj}L_{ip}(x_{ij}(x,\xi;s),\xi_{ij}(x,\xi;s))\right] ds
\end{align}
and for $1\leq i\leq m, 1\leq j\leq m$
\begin{multline}
	\Psi_{ij}[\mathbf{H}](x,\xi) = 
	\left(1-\delta_{ij}(x,\xi)\right) \frac{1}{\mu_j}\sum\limits_{r=1}^n \lambda_r q_{rj} \int_0^{s_{ir}^F(\chi_{ij}^F(x,\xi),0)} \left[\sum\limits_{k=1}^n\sigma^{++}_{kr}K_{ik}(x_{ir}(\chi_{ij}^F(x,\xi),0;s),\xi_{ir}(\chi_{ij}^F(x,\xi),0;s))\right.\\
	\left.+\sum\limits_{p=1}^m \sigma^{-+}_{pr}L_{ip}(x_{ir}(\chi_{ij}^F(x,\xi),0;s),\xi_{ir}(\chi_{ij}^F(x,\xi),0;s))\right] ds\\
		-\epsilon_{ij}\int_0^{\nu_{ij}^F(x,\xi)} \left[\sum\limits_{p=1}^m\sigma^{--}_{pj}L_{ip}\left(\chi_{ij}(x,\xi;\nu),\zeta_{ij}(x,\xi;\nu)\right)\right.+
		\left.\sum\limits_{k=1}^n \sigma^{+-}_{kj}K_{ik}\left(\chi_{ij}(x,\xi;\nu),\zeta_{ij}(x,\xi;\nu)\right)\right] d\nu.
\end{multline}
Define then the following sequence 
\begin{align}
	\mathbf{H}^0 (x,\xi) &= 0,\\
	\mathbf{H}^q(x,\xi) &= \boldsymbol{\phi}(x,\xi)+\boldsymbol{\Phi}[\mathbf{H}^{q-1}](x,\xi)\\
	&=\begin{pmatrix}
		\varphi_{11}(x,\xi)+\Phi_{11}[\mathbf{H}^{q-1}](x,\xi)\\ \vdots \\\varphi_{1n}(x,\xi) +\Phi_{1n}[\mathbf{H}^{q-1}](x,\xi)  \\ \varphi_{21}(x,\xi)+\Phi_{21}[\mathbf{H}^{q-1}](x,\xi)\\ \vdots \\ \varphi_{mn}(x,\xi)+\Phi_{mn}[\mathbf{H}^{q-1}](x,\xi) \\ \psi_{11} (x,\xi)+\Psi_{11}[\mathbf{H}^{q-1}](x,\xi) 
		\\ \vdots \\\psi_{mm}(x,\xi)+\Psi_{mm}[\mathbf{H}^{q-1}](x,\xi)
	\end{pmatrix}
\end{align}
One should notice that if the limit exists, then $\mathbf{H}=\lim\limits_{q\rightarrow+\infty}\mathbf{H}^q(x,\xi)$ is a solution of the integral equations, and thus solves the original hyperbolic system. Besides, define for $q\geq 1$ the increment $\Delta \mathbf{H}^q=\mathbf{H}^q-\mathbf{H}^{q-1}$, with $\Delta \mathbf{H}^0=\boldsymbol{\phi}$ by definition. Since the functional $\boldsymbol{\Phi}$ is linear, the following equation $\Delta \mathbf{H}^q(x,\xi)=\boldsymbol{\Phi}[\mathbf{H}^{q-1}](x,\xi)$ holds.  Using the definition of $\Delta \mathbf{H}^q$, it follows that if the sum $\sum\limits_{q=0}^{+\infty} \Delta \mathbf{H}^q(x,\xi)$ is finite, then
\begin{align}
	\mathbf{H}(x,\xi)&=\sum\limits_{q=0}^{+\infty} \Delta \mathbf{H}^q(x,\xi) \label{eq:Hsum}
\end{align}
In the next section, we prove convergence of the series in~${L}^\infty$.
\subsection{Convergence of the successive approximation series}
To prove convergence of the series, we look for a recursive upper bound, similarly to, e.g.~\cite{DiMeglio2013}. More precisely, let~$\epsilon$ be such that
\begin{align}
	0<\epsilon<1- \max_{1\leq j<i\leq m}\left\{\frac{\mu_i}{\mu_j} \right\}\label{eq:epsilon}.
\end{align}
Then, the following result holds
\begin{proposition}\label{prop:bounddeltaH}
	For $q\geq 1$, assume that
\begin{align}
	\forall (x,\xi)&\in\mathcal{T}, \; \forall i=1,...,nm+m^2  &\left|\Delta H_{i}(x,\xi) \right|&\leq \bar{\phi}\frac{M^q (x-(1-\epsilon)\xi)^q}{q!}
	\end{align}
	then, it follows that
	\begin{align}
	\forall (x,\xi)&\in\mathcal{T}, \; \forall i=1,...,m, \;\forall j=1,...,n, & \left|{\Phi_{ij}}[\Delta\mathbf{H}](x,\xi)\right|&\leq\bar{\phi}\frac{M^{q+1} (x-(1-\epsilon)\xi)^{q+1}}{(q+1)!}
	\end{align}
	\text{and }
	\begin{align}
	\forall (x,\xi)&\in\mathcal{T}, \; \forall i=1,...,m, \; \forall j=1,...,m, &\left|{\Psi_{ij}}[\Delta\mathbf{H}](x,\xi)\right|&\leq\bar{\phi}\frac{M^{q+1} (x-(1-\epsilon)\xi)^{q+1}}{(q+1)!}
\end{align}
\end{proposition}
The proof of this proposition relies on the following Lemma, which is crucial and different with previous works. 
\begin{Lemma}
	For $q\in \mathbb{N}$, $(x,\xi)\in\mathcal{T}$, and $s_{ij}^F(x,\xi)$, $\nu_{ij}^F(x,\xi)$, $x_{ij}(x,\xi,\cdot)$, $\xi_{ij}(x,\xi,\cdot)$, $\chi_{ij}(x,\xi,\cdot)$, $\zeta_{ij}(x,\xi,\cdot)$ defined as in~\eqref{eq:charK1},\eqref{eq:charK2},\eqref{eq:charL1},\eqref{eq:charL2}, respectively, the following inequalities holds
	
$\forall 1\leq i\leq m,$ $\forall 1\leq j\leq n$
\begin{align}
\int_0^{s_{ij}^F(x,\xi)} \left(x_{ij}(x,\xi;s)-(1-\epsilon)\xi_{ij}(x,\xi;s)\right)^qds\leq M_\lambda \frac{\left(x-(1-\epsilon)\xi\right)^{q+1}}{q+1}\label{eq:ineqxxi}
\end{align}
$\forall {1\leq i,j\leq m}$
\begin{align}
\int_0^{\nu_{ij}^F(x,\xi)} \left(\chi_{ij}(x,\xi;\nu)-(1-\epsilon)\zeta_{ij}(x,\xi;\nu)\right)^qd\nu
\leq M_\lambda \frac{\left(x-(1-\epsilon)\xi\right)^{q+1}}{q+1}\label{eq:ineqchizeta}
\end{align}
where
\begin{align}
	M_\lambda = 
	\max\limits_{i,p=1,...,m,j=1,...,n} \left\{\frac{1}{\mu_i+(1-\epsilon)\lambda_{j}} , \frac{1}{-\epsilon_{ij}\left(\mu_i-(1-\epsilon)\mu_{p}\right)}\right\}
\end{align}
\end{Lemma}
\begin{proof}
	Consider the following change of variables, noting \eqref{eq:charK1},\eqref{eq:charK2},
	\begin{align}
			\tau&=x_{ij}(x,\xi;s)-(1-\epsilon)\xi_{ij}(x,\xi;s),\\
			d\tau &= \left[\frac{dx_{ij}}{ds}(x,\xi;s)-(1-\epsilon)\frac{d\xi_{ij}}{ds}(x,\xi;s)\right]ds\\
			&=\left(-\mu_ i - (1-\epsilon)\lambda_j\right) ds
	\end{align}
	The left-hand-side of~\eqref{eq:ineqxxi} becomes
\begin{align}
	\int_0^{s_{ij}^F(x,\xi)} \left(x_{ij}(x,\xi;s)-(1-\epsilon)\xi_{ij}(x,\xi;s)\right)^qds
	&=\int_{x-(1-\epsilon)\xi}^{x_{ij}^F(x,\xi)-(1-\epsilon)\xi^F_{ij}(x,\xi)} \frac{-\tau^q}{\mu_ i + (1-\epsilon)\lambda_j}d\tau\\
	\\
	&=	\frac{\left(x-(1-\epsilon)\xi\right)^{q+1}-\left(x_{ij}^F(x,\xi)-(1-\epsilon)\xi^F_{ij}(x,\xi)\right)^{q+1}}{(\mu_ i + (1-\epsilon)\lambda_j)(q+1)} \\
	&\leq  M_\lambda \frac{\left(x-(1-\epsilon)\xi\right)^{q+1}}{q+1}
\end{align}
where we have used the fact that for all~$1\leq i \leq m$, $1 \leq j \leq n$, one has
\begin{align}
	x_{ij}^F(x,\xi)-(1-\epsilon)\xi^F_{ij}(x,\xi)\geq 0
\end{align}
which is trivially satisfied since~$(x_{ij}^F(x,\xi),\xi^F_{ij}(x,\xi)) \in \partial \mathcal{T}$ and~$\epsilon>0$. Consider now the following change of variables
\begin{align}
	\tau &= \chi_{ij}(x,\xi;s)-(1-\epsilon)\zeta_{ij}(x,\xi;s),\\
	d\tau &= \left[\frac{d\chi_{ij}}{ds}(x,\xi;s)-(1-\epsilon)\frac{d\zeta_{ij}}{ds}(x,\xi;s)\right]ds\\
			&=\epsilon_{ij}\left(\mu_ i - (1-\epsilon)\mu_j\right) ds 
\end{align}
Thus, the left-hand-side of~\eqref{eq:ineqchizeta} becomes
\begin{align}
	\int_0^{\nu_{ij}^F(x,\xi)} \left(\chi_{ij}(x,\xi;\nu)-(1-\epsilon)\zeta_{ij}(x,\xi;\nu)\right)^qd\nu
	&=\int_{x-(1-\epsilon)\xi}^{\chi_{ij}^F(x,\xi)-(1-\epsilon)\zeta^F_{ij}(x,\xi)} \frac{\tau^q}{\epsilon_{ij}\left(\mu_ i-(1-\epsilon)\mu_j\right)}d\tau\\
	&=	\frac{\left(x-(1-\epsilon)\xi\right)^{q+1}-\left(\chi_{ij}^F(x,\xi)-(1-\epsilon)\zeta^F_{ij}(x,\xi)\right)^{q+1}}{-\epsilon_{ij}\left( \mu_ i - (1-\epsilon)\mu_j\right)(q+1)}\label{eq:ineqintchizeta}
\end{align}
Given the definition of~$\epsilon_{ij}$ given by~\eqref{eq:epsilonij}, one has
\begin{align}
	-\epsilon_{ij}\left( \mu_ i - (1-\epsilon)\mu_j\right) = \begin{cases}
			\mu_ i - (1-\epsilon)\mu_j & \text{if } i\leq j\\
			(1-\epsilon)\mu_j - \mu_ i & \text{if } i>j
	\end{cases}
\end{align}
Therefore, given the definition of~$\epsilon$ (Equation~\eqref{eq:epsilon}) in the case~$i> j$ and the ordering of the~$\mu_i$ in the case~$i\leq j$, one has
\begin{align}
	-\epsilon_{ij}\left( \mu_ i - (1-\epsilon)\mu_j\right)>0 \label{eq:ineqepsilon}
\end{align}
Besides, since~$(\chi_{ij}^F(x,\xi),\zeta_{ij}^F(x,\xi))\in \mathcal{T}$, one has~$\left(\chi_{ij}^F(x,\xi)-(1-\epsilon)\zeta^F_{ij}(x,\xi)\right)>0$ and~\eqref{eq:ineqintchizeta} becomes
\begin{align}
	\int_0^{\nu_{ij}^F(x,\xi)} \left(\chi_{ij}(x,\xi;\nu)-(1-\epsilon)\zeta_{ij}(x,\xi;\nu)\right)^qd\nu \leq  M_\lambda \frac{\left(x-(1-\epsilon)\xi\right)^{q+1}}{q+1}
\end{align}
which concludes the proof. 
\end{proof}
\begin{rmk}
	Notice that~\eqref{eq:ineqepsilon} also implies that, for any~$(x,\xi)\in \mathcal{T}$ and~$1\leq i \leq m$, $1 \leq j \leq n$ the function
	\begin{align}
		\nu \in [0,\nu^F_{ij}(x,\xi)] \mapsto \chi_{ij}(x,\xi;\nu)-(1-\epsilon)\zeta_{ij}(x,\xi;\nu)
	\end{align}
	is strictly decreasing, in particular the following inequality holds
	\begin{align}
		0\leq \chi^F_{ij}(x,\xi)-(1-\epsilon)\zeta^F_{ij}(x,\xi) \leq x-(1-\epsilon)\xi \label{eq:chizeta}
	\end{align}
	which will be useful in the proof of Proposition~\ref{prop:bounddeltaH}.
\end{rmk}
\begin{proof}[Proof of Proposition~\ref{prop:bounddeltaH}]
	Define
	\begin{align}
			\bar{\lambda} &= \max \left\{\lambda_n,\mu_1\right\},\quad \underline{\lambda} = \max\left\{\frac{1}{\lambda_1},\frac{1}{\mu_n}\right\},\\
			\bar{\sigma} &= \max\limits_{i,j} \left\{\sigma^{++},\sigma^{-+},\sigma^{+-},\sigma^{--}\right\},\quad \bar{q}= \max\limits_{i,j} \{q_{ij}\} \\
			M&=\left(n \bar{\lambda}\underline{\lambda}\bar{q}+1\right)(n+m)\bar{\sigma}M_\lambda,\\
			\bar{\phi}&=\max\limits_{i,j}\max\limits_{(x,\xi)\in\mathcal{T}} \left\{|\varphi_{i,j}(x,\xi)|,\ |\psi_{i,j}(x,\xi)|\right\}
	\end{align}
	Let now~$q\in\mathbb{N}$ and assume that 
	\begin{align}
	\forall &(x,\xi)\in\mathcal{T}, \; \forall i=1,...,nm+m^2 & \left|\Delta H_{i}(x,\xi) \right|&\leq \bar{\phi}\frac{M^q (x-(1-\xi))^q}{q!}\label{eq:upperboundDeltaH}	
	\end{align}
	Then, for~$1\leq i\leq m$, $1 \leq j \leq n$, $(x,\xi)\in \mathcal{T}$ one has
	\begin{multline}
			\left|{\Phi_{ij}}[\Delta\mathbf{H}](x,\xi)\right|\leq \int_0^{s_{ij}^F(x,\xi)} \left|\sum\limits_{k=1}^n\sigma^{++}_{kj}\Delta K_{ik}(x_{ij}(x,\xi;s),\xi_{ij}(x,\xi;s))\right.\\\left.+\sum\limits_{p=1}^m \sigma^{-+}_{pj}\Delta L_{ip}(x_{ij}(x,\xi;s),\xi_{ij}(x,\xi;s))\right| ds
			\end{multline}
			using~\eqref{eq:ineqxxi} and~\eqref{eq:upperboundDeltaH}, this yields
			\begin{align}
	\left|{\Phi_{ij}}[\Delta\mathbf{H}](x,\xi)\right|&\leq \ (n+m)\bar{\sigma}
	\cdot\int_0^{s_{ij}^F(x,\xi)} \bar{\phi}\frac{M^q\left(x_{ij}(x,\xi;s)-(1-\epsilon)\xi_{ij}(x,\xi;s)\right)^q}{q!} ds\\
	&\leq \ (n+m)\bar{\sigma}\frac{\bar{\phi}M^q}{q!}M_\lambda \frac{(x-(1-\xi))^{q+1}}{q+1}\\
	&\leq \ \bar{\phi} \frac{M^{q+1} (x-(1-\epsilon)\xi)^{q+1}}{(q+1)!}
	\end{align}
	Similarly, for~$1 \leq i,j \leq m$, one gets, using
	~\eqref{eq:upperboundDeltaH}
	\begin{multline}
		\left|{\Psi_{ij}}[\Delta\mathbf{H}](x,\xi)\right|\leq 		\bar{\lambda}\underline{\lambda}\bar{q}(n+m)\bar{\sigma}\sum\limits_{r=1}^n \int_0^{s_{ir}^F(\chi_{ij}^F(x,\xi),0)} 
		\bar{\phi}\frac{M^q\left(x_{ir}(\chi_{ij}^F(x,\xi),0;s)-(1-\epsilon)\xi_{ir}(\chi_{ij}^F(x,\xi),0;s)\right)^q}{q!}ds\\
		+(n+m)\bar{\sigma}\int_0^{\nu_{ij}^F(x,\xi)}
		\bar{\phi}\frac{M^q\left(\chi_{ij}(x,\xi;\nu)-(1-\epsilon)\zeta_{ij}(x,\xi;\nu)\right)^q}{q!}d\nu
	\end{multline}
	Then, using~\eqref{eq:ineqxxi} at~$(x,\xi)=(\chi_{ij}^F(x,\xi),0)$ and~\eqref{eq:ineqchizeta} yields
	\begin{multline}
\left|{\Psi_{ij}}[\Delta\mathbf{H}](x,\xi)\right|
		\leq  \bar{\lambda}\underline{\lambda}\bar{q}(n+m)\bar{\sigma}n\bar{\phi}M_\lambda M^q\frac{\left(\chi_{ij}^F(x,\xi)-(1-\epsilon)\zeta_{ij}^F(x,\xi)\right)^{q+1}}{(q+1)!}\\
		+(n+m)\bar{\sigma} \bar{\phi}\frac{M^qM_\lambda (x-(1-\epsilon)\xi)^{q+1}}{(q+1)!}	
		\end{multline}
		{Inequality~\eqref{eq:chizeta} yields}
		\begin{align}
		\left|{\Psi_{ij}}[\Delta\mathbf{H}](x,\xi)\right|&\leq \left(n \bar{\lambda}\underline{\lambda}\bar{q}+1\right)(n+m)\bar{\sigma} \bar{\phi}M_\lambda \frac{M^q(x-(1-\epsilon)\xi)^{q+1}}{(q+1)!}\\
		&\leq \bar{\phi}\frac{M^{q+1}(x-(1-\epsilon)\xi)^{q+1}}{(q+1)!}
	\end{align}
	which concludes the proof.
\end{proof}
Proposition~\ref{prop:bounddeltaH} directly leads to Theorem \ref{the:wellposednessKandL}, since by the same procedures presented in \cite{Coron2013} and \cite{DiMeglio2013}, one has that \eqref{eq:Hsum} converges and 
\begin{align}
\left|\mathbf{H}(x,\xi)\right|&=\left|\sum\limits_{q=0}^{+\infty} \Delta \mathbf{H}^q(x,\xi)\right|\leq \bar{\phi}\text{e}^{M(x-(1-\epsilon)\xi)}.\label{eq:kernelbound}
\end{align}

\section{Concluding remarks}\label{sec:outlook}
We have presented boundary control designs for a general class of linear first-order hyperbolic systems: an output-feedback law for stabilization of heterodirectional systems and a tracking controller for motion planning for homodirectional systems. 

These results bridge the gap with the results of, e.g.~\cite{Li2010}, where the null (or weak) controllability of~$(n+m)$--state heterodirectional states is proved but no explicit design is given. 

Our results open the door for a large number of related problems to be solved, e.g. collocated observer design, disturbance rejection, similarly to~\cite{Aamo2012}, parameter identification as in~\cite{DiMeglio2014}, output-feedback adaptive control as in~\cite{berkrs2014}, and stabilization of quasilinear systems as in~\cite{Coron2013}. 

Another important question concerns the degree of freedom given by Equation~\eqref{eq:artificialboundary} in the control design. The effect of the boundary value of the kernels on the transient performances of the closed-loop system is non-trivial, yet crucical for applications.

\section*{Acknowledgements}
The authors would like to thank Jean-Michel Coron for his encouragement and fruitful discussions.

\bibliographystyle{plain}

\begin{thebibliography}{10}

\bibitem{Aamo2012}
O.M.~Aamo.
\newblock Disturbance rejection in 2x2 linear hyperbolic systems.
\newblock {\em Automatic Control, IEEE Transactions on}, PP(99):1, 2012.

\bibitem{Amin2008}
S.~Amin, F. Hante, and A.~Bayen.
\newblock {\em Hybrid Systems: Comp}, chapter On stability of switched linear
  hyperbolic conservation laws with reflecting boundaries, pages 602--605.
\newblock Springer-Verlag, 2008.

\bibitem{berkrs2014}
P. Bernard and M. Krstic. \newblock Adaptive output-feedback stabilization of non-local hyperbolic PDEs. \newblock{Automatica}, vol. 50, pp. 2692--2699, 2014.

\bibitem{Coron2007}
J.-M. Coron.
\newblock {\em Control and Nonlinearity}.
\newblock American Mathematical Society, 2007.

\bibitem{Coron2008}
J.-M. Coron, G.~Bastin, and B.~{d'Andr{\'e}a-Novel}.
\newblock {Dissipative boundary conditions for one-dimensional nonlinear
  hyperbolic systems}.
\newblock {\em SIAM Journal on Control and Optimization}, 47(3):1460--1498,
  2008.

\bibitem{Coron1999}
J.-M. Coron, B.~{d'Andr{\'e}a-Novel}, and G.~Bastin.
\newblock A lyapunov approach to control irrigation canals modeled by
  saint-venant equations.
\newblock {\em Proceedings of the 1999 European Control Conference, Karlsruhe,
  Germany}, 1999.

\bibitem{Coron2013}
J.-M. Coron, R.~Vazquez, M.~Krstic, and G.~Bastin.
\newblock Local exponential $h^2$ stabilization of a $2\times 2$ quasilinear
  hyperbolic system using backstepping.
\newblock {\em SIAM Journal on Control and Optimization}, 51(3):2005--2035,
  2013.

\bibitem{Coron2014}
Jean-Michel Coron and Georges Bastin.
\newblock Dissipative boundary conditions for one-dimensional quasi-linear
  hyperbolic systems: Lyapunov stability for the {$C^1$}-norm.
\newblock {\em Preprint}, 2014.

\bibitem{Halleux2003}
J.~de~Halleux, C.~Prieur, J.-M. Coron, B.~{d'Andr{\'e}a-Novel}, and G.~Bastin.
\newblock Boundary feedback control in networks of open channels.
\newblock {\em Automatica}, 39(8):1365 -- 1376, 2003.

\bibitem{DiMeglio2011}
F.~Di~Meglio.
\newblock {\em Dynamics and control of slugging in oil production}.
\newblock PhD thesis, MINES ParisTech, 2011.

\bibitem{DiMeglio2014}
F.~Di~Meglio, D.~Bresch-Pietri, and U.~J.~F. Aarsnes.
\newblock An adaptive observer for hyperbolic systems with application to
  underbalanced drilling.
\newblock In {\em Proceeding of the 2014 IFAC World Congress}, pages
  11391--11397, 2014.

\bibitem{DiMeglio2013}
F.~Di~Meglio, R.~Vazquez, and M.~Krstic.
\newblock Stabilization of a system of $n+1$ coupled first-order hyperbolic
  linear pdes with a single boundary input.
\newblock {\em Automatic Control, IEEE Transactions on}, 58(12):3097--3111,
  2013.

\bibitem{Diagne2012}
A.~Diagne, G.~Bastin, and J.-M. Coron.
\newblock Lyapunov exponential stability of 1-d linear hyperbolic systems of
  balance laws.
\newblock {\em Automatica}, 48(1):109 -- 114, 2012.

\bibitem{Djordjevic2010}
S.~Djordjevic, O.H. Bosgra, P.M.J. Van~den Hof, and D.~Jeltsema.
\newblock Boundary actuation structure of linearized two-phase flow.
\newblock In {\em American Control Conference (ACC), 2010}, pages 3759 --3764,
  30 2010-july 2 2010.

\bibitem{Dudret2012}
S.~Dudret, K.~Beauchard, F.~Ammouri, and Rouchon P.
\newblock Stability and asymptotic observers of binary distillation processes
  described by nonlinear convection/diffusion models.
\newblock {\em Proceedings of the 2012 American Control Conference, Montr\'eal,
  Canada}, pages 3352--3358, 2012.

\bibitem{Greenberg1984}
J.~M. Greenberg and Ta~Tsien Li.
\newblock The effect of boundary damping for the quasilinear wave equation.
\newblock {\em J. Differential Equations}, 52(1):66--75, 1984.

\bibitem{Gugat2011}
Martin Gugat, Markus Dick, and G{\"u}nter Leugering.
\newblock Gas flow in fan-shaped networks: classical solutions and feedback
  stabilization.
\newblock {\em SIAM J. Control Optim.}, 49(5):2101--2117, 2011.

\bibitem{Gugat2011a}
M. Gugat and M. Herty.
\newblock Existence of classical solutions and feedback stabilization for the
  flow in gas networks.
\newblock {\em ESAIM Control Optim. Calc. Var.}, 17(1):28--51, 2011.

\bibitem{Hochstadt1973}
H. Hochstadt.
\newblock Integral Equations.
\newblock {\em Wiley-Interscience, New York}, 1973

\bibitem{Hu2015}
L.~Hu and F.~Di~Meglio.
\newblock Finite-time backstepping stabilization of $3 \times 3$ hyperbolic
  systems.
\newblock {\em IEEE European Control Conference}, under review, 2015.

\bibitem{Li1994}
Ta~Tsien Li.
\newblock {\em Global classical solutions for quasilinear hyperbolic systems},
  volume~32 of {\em RAM: Research in Applied Mathematics}.
\newblock Masson, Paris; John Wiley \& Sons, Ltd., Chichester, 1994.

\bibitem{Li2010}
Tatsien Li and Bopeng Rao.
\newblock Strong (weak) exact controllability and strong (weak) exact
  observability for quasilinear hyperbolic systems.
\newblock {\em Chinese Annals of Mathematics, Series B}, 31(5):723--742, 2010.

\bibitem{Qin1985}
Tie~Hu Qin.
\newblock Global smooth solutions of dissipative boundary value problems for
  first order quasilinear hyperbolic systems.
\newblock {\em Chinese Ann. Math. Ser. B}, 6(3):289--298, 1985.
\newblock A Chinese summary appears in Chinese Ann. Math. Ser. A {{\bf{6}}}
  (1985), no. 4, 514.

\bibitem{VazKrs2010}
R. Vazquez and M. Krstic.
\newblock Motion planning for homodirectional $2\times2$ hyperbolic linear PDEs.
Presented at \newblock {\em Symposium on Analysis and Control of Infinite-Dimensional Systems in the Engineering Sciences}, Max Planck Institute for Dynamics of Complex Technical Systems, Magdeburg, Germany, 2010.

\bibitem{Vazquez2014}
R. Vazquez and M. Krstic.
\newblock Marcum {Q}-functions and explicit kernels for stabilization of linear hyperbolic systems with constant coefficients.
\newblock {\em Systems \& Control Letters }, 68:33--42, 2014

\bibitem{Xu2002}
C.-Z. Xu and G.~Sallet.
\newblock Exponential stability and transfer functions of processes governed by
  symmetric hyperbolic systems.
\newblock {\em ESAIM: Control, Optimisation and Calculus of Variations},
  7:421--442, 2002.

\end{thebibliography}

\renewcommand\thesection{\Alph{section}}
\renewcommand\theequation{A.\arabic{equation}}
\setcounter{equation}{0}
\setcounter{section}{0}

\end{document}